\documentclass[10pt,a4paper,american]{amsart}
\usepackage{babel}
\usepackage{amsthm}
\usepackage{amsbsy}
\usepackage{amstext}
\usepackage{amsfonts}
\usepackage{amsmath}
\usepackage{mathcomp}
\usepackage{mathtools}
\usepackage{dsfont}
\usepackage{latexsym}
\usepackage{amssymb}
\usepackage{esint}	
\usepackage[colorlinks,pdfpagelabels,pdfstartview = FitH,bookmarksopen = true,bookmarksnumbered = true,linkcolor = blue,plainpages = false,hypertexnames = false,citecolor = red,pagebackref=false]{hyperref}
\usepackage{cite}

\newtheorem{theo}{Theorem}[section]
\newtheorem{cor}[theo]{Corollary}
\newtheorem{prop}[theo]{Proposition}
\newtheorem{lem}[theo]{Lemma}

\theoremstyle{definition}
\newtheorem{defin}[theo]{Definition}
\newtheorem*{lem*}{Lemma}
\newtheorem{rem}[theo]{Remark}

\newtheorem*{cor*}{Corollary}
\newtheorem*{theo*}{Theorem}

\DeclareMathOperator*{\essliminf}{ess\,lim\,inf}
\DeclareMathOperator*{\esssup}{ess\,sup}
\DeclareMathOperator*{\essinf}{ess\,inf}

\newcommand{\N}{\ensuremath{\mathbb{N}}}
\newcommand{\R}{\ensuremath{\mathbb{R}}}

\newcommand{\spt}{\operatorname{spt}}

\newcommand{\mollifytime}[2]{[\![ #1 ]\!]_{#2}}

\makeatletter
\@namedef{subjclassname@2020}{%
  \textup{2020} Mathematics Subject Classification}
\makeatother

\def\Xint#1{\mathchoice
    {\XXint\displaystyle\textstyle{#1}}%
    {\XXint\textstyle\scriptstyle{#1}}%
    {\XXint\scriptstyle\scriptscriptstyle{#1}}%
    {\XXint\scriptscriptstyle\scriptscriptstyle{#1}}%
    \!\int}
\def\XXint#1#2#3{\setbox0=\hbox{$#1{#2#3}{\int}$}
    \vcenter{\hbox{$#2#3$}}\kern-0.5\wd0}
\def\bint{\Xint-}
\def\dashint{\Xint{\raise4pt\hbox to7pt{\hrulefill}}}

\def\XXiint#1#2#3{\setbox0=\hbox{$#1{#2#3}{\iint}$}
    \vcenter{\hbox{$#2#3$}}\kern-0.5\wd0}

\renewcommand{\epsilon}{\varepsilon}
\newcommand{\eps}{\varepsilon}
\renewcommand{\rho}{\varrho}
\newcommand{\ph}{\varphi}

\renewcommand{\epsilon}{\varepsilon}
\renewcommand{\rho}{\varrho}
\renewcommand{\d}{\:\! \mathrm{d}}

\DeclareMathOperator{\loc}{loc}

\numberwithin{equation}{section}
\allowdisplaybreaks

\begin{document}
\renewcommand{\refname}{References} 
\renewcommand{\abstractname}{Abstract} 
\title[Supercaloric functions for the PME]{Supercaloric functions for the porous medium equation in the fast diffusion case}

\author[K. Moring]{Kristian Moring}
\address{Kristian Moring\\
	Fakult\"at f\"ur Mathematik, Universit\"at Duisburg-Essen\\
	Thea-Leymann-Str. 9, 45127 Essen, Germany}
\email{kristian.moring@uni-due.de}

\author[C. Scheven]{Christoph Scheven}
\address{Christoph Scheven\\
	Fakult\"at f\"ur Mathematik, Universit\"at Duisburg-Essen\\
	Thea-Leymann-Str. 9, 45127 Essen, Germany}
\email{christoph.scheven@uni-due.de}


\begin{abstract}
We study a generalized class of supersolutions, so-called super\-caloric
functions to the porous medium equation in the fast diffusion
case. Supercaloric functions are defined as lower semicontinuous
functions obeying a parabolic comparison principle. We prove that
bounded supercaloric functions are weak supersolutions. In the
supercritical range, we show that unbounded supercaloric functions can
be divided into two mutually exclusive classes dictated by the
Barenblatt solution and the infinite point-source solution, and give
several characterizations for these classes. Furthermore, we study the
pointwise behavior of supercaloric functions and obtain connections between supercaloric functions and weak supersolutions.
\end{abstract}
\makeatother

\maketitle

\section{Introduction}

In this paper we study supersolutions to the porous medium equation (PME for short), which can be written as 
\begin{equation}\label{evo_eqn}
\partial_t u -\Delta( u^m)=0, 
\end{equation}  
for $0<m<\infty$ and nonnegative $u$. We are concerned with the fast diffusion range $0<m<1$, and in particular, in some of the main results in the supercritical fast diffusion range $\frac{n-2}{n} < m < 1$. Furthermore, we suppose that the spatial dimension satisfies $n \geq 2$. For the standard theory of the porous medium equation we refer to the monographs~\cite{Vazquez, Vazquez2, DK}.

The theory of supercaloric functions for the parabolic $p$-Laplace
equation in the supercritical case is well developed. In the slow
diffusion case, Sobolev space properties of locally bounded
supercaloric functions were proven in~\cite{KinnunenLindqvist2006},
and the classification theory of unbounded supercaloric functions is summarized in~\cite{KuLiPa}. In~\cite{KoKuPa}, the study of bounded supercaloric
functions was extended to the supercritical fast diffusion range, and
for the classification theory in this case for unbounded supercaloric functions we refer to~\cite{GKM_supercal}.

For the porous medium equation the analogous theory in the slow
diffusion case is well established. Sobolev space properties of
supercaloric functions were studied in~\cite{KinnunenLindqvist_crelle},
and for the classification theory in the unbounded case we refer
to~\cite{KLLP-supercal}. The theory in the fast diffusion range is currently open, which we address in this paper. To our knowledge, many questions in the critical and subcritical cases are still open for both equations, which are left to subjects of future research.

The structure of the porous medium equation poses some well-known
challenges. For example, solutions are not closed under addition or
multiplication by constants. In our case, the former poses a serious
difficulty in obtaining an appropriate Caccioppoli inequality and
comparison principles, for example. A critical feature that occurs is
that one can not approximate nonnegative solutions with strictly
positive ones by adding constants, and in this way avoid the set
$\{u=0\}$ where the equation becomes singular. In order to overcome
this difficulty, we are able to show that in each connected component
of the domain every supercaloric function is either strictly positive
or vanishes identically on any given time-slice, see Lemma~\ref{lem:alternatives}. The proof of this
property relies on an expansion of positivity result for weak solutions (see~\cite{DGV}), which holds in the whole fast diffusion range $0<m<1$. Furthermore, this allows us to express the set where a supercaloric function is strictly positive as a countable union of time intervals in every connected component of the domain. The described phenomenon is strongly tied to the nature of fast diffusion, and it does not occur as such in the slow diffusion case.

In Section~\ref{s.connection} we show that the class of locally
bounded supercaloric functions is included in the class of weak
supersolutions; a result which was shown for the parabolic $p$-Laplace
equation in~\cite{KoKuPa,GKM_supercal,KinnunenLindqvist2006} and for
the porous medium equation in the slow diffusion case in~\cite{KinnunenLindqvist_crelle}. 
The proof is roughly divided into two parts. First, the result is
shown for strictly positive supercaloric functions in
Lemma~\ref{l.bdd-positive-supercal-is-supersol}, whose proof relies on
a suitable obstacle problem stated in Theorem~\ref{t.obstacle-supercal}, which is based on the results in~\cite{BLS,Schaetzler2,Cho_Scheven,MS,MSb}. In the second step, this result is generalized to hold for nonnegative supercaloric functions (Theorem~\ref{t.bdd_super}). The geometry of positivity sets of supercaloric functions established in Section~\ref{sec:positivity} plays an important role in the second part of the proof.

In the supercritical case, we show that supercaloric functions can be
divided into two mutually exclusive classes, which we call the Barenblatt
class and the complementary class. The former is modeled by the
Barenblatt solution~\eqref{e.barenblatt2}, while the latter is modeled by so-called infinite
point-source solution~\eqref{e.IPSS}, see~\cite{CV}. Functions in the Barenblatt
class have some regularity properties, e.g. in terms of integrability
(Theorem~\ref{t.barenblatt}), while functions in the complementary
class are not guaranteed to have any
(Theorem~\ref{t.complementary}). As was noticed already in the case of
the parabolic $p$-Laplace equation
(\hspace{1sp}\cite{KuLiPa,GKM_supercal}), prominent singularities of
functions in the complementary class are qualitatively different in
the fast diffusion case than in the slow diffusion case
(\hspace{1sp}\cite{KLLP-supercal}). Roughly speaking, variables in
space and time change their roles in this respect. For Sobolev space
properties in the Barenblatt class we use a Moser type iteration,
which is based on the combination of Sobolev inequality and a suitable Caccioppoli inequality. On the other hand, proofs in the complementary class are based on Harnack type inequalities stated in Section~\ref{sec:positivity}.



In the final section we study the pointwise behavior of supercaloric
functions. It is well known that every weak supersolution is lower
semicontinuous after possible redefinition in a set of measure zero,
see~\cite{Naian,AL}. More precisely, pointwise values can be recovered
almost everywhere by the $\essliminf$ of the function, where only
instances of time in the past are relevant. For the parabolic
$p$-Laplace equation it was shown in~\cite{KinnunenLindqvist2006}, and
for the porous medium equation in the slow diffusion case
in~\cite{KinnunenLindqvist_crelle} that supercaloric functions enjoy
the same property at every point in their domain (for the elliptic
case, see also~\cite{HKM}). In Section~\ref{sec:pointwise} we show
that the same property holds for supercaloric functions to the porous
medium equation in the fast diffusion case. We conclude the paper by
summarizing the connections between supercaloric functions and weak
supersolutions in Corollary~\ref{c.connections}.

\medskip
 
\noindent
{\bf Acknowledgments.} K.~Moring has been supported by the Magnus Ehrnrooth Foundation.

\section{Weak supersolutions} \label{sec:prelim}

Let $\Omega \subset \R^n$ be an open set. For $T>0$ we denote
by $\Omega_T := \Omega \times (0,T)$ a space-time cylinder in
$\R^{n+1}$. The parabolic boundary of $\Omega_T$ is defined as
$\partial_p \Omega_T := \left( \Omega \times \{0\} \right) \cup \left(
  \partial \Omega \times [0,T) \right)$. We call $\Omega_T$ a
$C^{k,\alpha}$-cylinder if $\Omega\subset\R^n$ is a bounded
$C^{k,\alpha}$-domain for $k\in\N$ and $\alpha > 0$.

\subsection{Notion of weak solutions}
We begin by defining the concept of weak (super- and sub)solutions.
\begin{defin} \label{d.weak_sol}
A measurable function $u: \Omega_T \to [0,\infty]$ satisfying
$$
u^m \in L_{\loc}^2(0,T;H^{1}_{\loc}(\Omega)) \cap L^\frac{1}{m}_{\loc}(\Omega_T)
$$
is called a weak solution to the PME~\eqref{evo_eqn} if and only if $u$ satisfies the integral equality
\begin{align}
\iint_{\Omega_T} \left(-u \partial_t \varphi + \nabla u^m \cdot \nabla \varphi \right)\d x\d t  =0
\end{align}
for every $\varphi \in C^\infty_0(\Omega_T)$. Further, we say that $u$
is a weak supersolution if the integral above is nonnegative for all
nonnegative test functions $\varphi \in C^{\infty}_0(\Omega_T)$. If
the integral is nonpositive for such test functions, we call $u$ a
weak subsolution.

Finally, we say that $u:\Omega_T\to[0,\infty]$ is a global weak
solution to the PME~\eqref{evo_eqn} if it is a weak solution with the
property
\begin{equation*}
  u^m\in L^2(0,T;H^{1}(\Omega)) \cap L^{\frac{1}{m}}(\Omega_T).
\end{equation*}

\end{defin}

Then we recall a comparison principle for weak super(sub)solutions, see~\cite{Bjorns_boundary,Vazquez,DK}.

\begin{lem} \label{l.comparison-weaksupersub}
Let $0 < m < 1$ and $\Omega_T$ be a $C^{2,\alpha}$-cylinder with $\alpha>0$. Suppose that $u$ is a weak supersolution and $v$ a weak subsolution to~\eqref{evo_eqn} in $\Omega_T$, such that $u^m,v^m \in L^2(0,T;H^1(\Omega)) \cap L^\frac{2}{m}(\Omega_T)$. If, in addition
$$
(v^m-u^m)_+ (\cdot,t) \in H^{1}_0(\Omega), \, \text{ for a.e. } t\in (0,T),
$$
and 
$$
\lim_{h\to 0} \frac{1}{h} \int_{0}^{h} \int_\Omega (v-u)_+ \, \d x \d t = 0
$$
holds true, then $0\leq v\leq u$ a.e. in $\Omega_T$.
\end{lem}

The following maximum principle also holds, see~\cite{MSb}.

\begin{lem} \label{l.maximum-principle}
  Let $m > 0$. Let $u$ be a weak subsolution with the property
  $u^m \in L^2(0,T;H^1(\Omega)) \cap L^\frac{1}{m}(\Omega_T)$ and $k \in \R_{\geq0}$. If $(u^m - k^m)_+(\cdot,t) \in H^1_0(\Omega)$ for a.e. $t \in (0,T)$ and 
$$
\lim_{h\to 0} \frac{1}{h} \int_0^h \int_\Omega (u-k)_+ \, \d x \d t = 0,
$$
then 
$$
u \leq k \quad \text{ a.e. in } \Omega_T.
$$
\end{lem}

Even though we cannot add constants to solutions, we can show the following result for weak solutions with perturbed boundary values. For the proof in the case $m > 1$, see~\cite[Lemma 3.2]{KLL}.

\begin{lem} \label{l.u_eps-u-estimate}
Suppose that $0<m<1$ and $\Omega\Subset\R^n$. Let $g$ be a nonnegative function satisfying $g^m \in
L^2(0,T;H^1(\Omega))$, $g \in C([0,T]; L^{m+1}(\Omega)) \cap L^\infty
(\Omega_T)$.  Denote $g_\eps = (g^m +
\eps^m)^\frac{1}{m}$, for $\eps\in(0,1]$. Let $u$ and $u_\eps$ be global weak solutions in $\Omega_T$ \textup{(}in class $C([0,T];L^{m+1}(\Omega))$\textup{)}, taking boundary values $g$ and $g_\eps$, respectively, in the Sobolev sense on the lateral boundary, and $u (x,0) = g(x,0)$ and $u_\eps(x,0) = g_\eps(x,0)$ for a.e. $x \in \Omega$. Then, there exists $c = c(m,\|g\|_\infty,|\Omega|,T) > 0$ such that
$$
\iint_{\Omega_T} (u_\eps - u)(u_\eps^m - u^m) \, \d x \d t \leq c \delta (\eps), 
$$
in which $\delta (\eps) := \max \left\{\eps^m, \int_\Omega (g_\eps(x,0) - g(x,0))\, \d x \right\} \xrightarrow{\eps \to 0} 0$. 
\end{lem}

\begin{proof}
We use the Oleinik type test function 
\[
\eta (x,t) := 
\begin{cases}
\int_t^T (u_\eps^m -u^m  - \eps^m) \, \d s, & \mbox{for }0<t<T,\\
0, & \mbox{for }t \geq T,
\end{cases}
\]
in the weak formulation. Observe that this function vanishes on the lateral boundary in Sobolev sense, and 
$$
\partial_t \eta = - (u_\eps^m - u^m) + \eps^m,\quad \nabla \eta =
\int_t^T \nabla (u_\eps^m - u^m) \, \d s 
\quad\mbox{on $\Omega_T$. }
$$
By subtracting the weak formulations with the given test function, we obtain
\begin{align*}
\iint_{\Omega_T} (u_\eps - u)(u_\eps^m - u^m - &\eps^m) + \nabla (u_\eps^m - u^m)\cdot \int_t^T \nabla (u_\eps^m - u^m)\, \d s\,  \d x \d t \\
&= \int_\Omega (g_\eps(x,0) - g(x,0)) \int_0^T (u_\eps^m - u^m - \eps^m) \, \d s \, \d x \\
&= \int_\Omega (g_\eps(x,0) - g(x,0)) \int_0^T (u_\eps^m - u^m ) \, \d s \, \d x \\
&\phantom{+} - \eps^m T \int_\Omega (g_\eps(x,0) - g(x,0)) \, \d x .
\end{align*}
The divergence part on the left-hand side equals
$$
\frac{1}{2} \int_\Omega \left( \int_0^T (\nabla u_\eps^m - \nabla u^m ) \, \d t \right)^2 \, \d x \geq 0,
$$
such that we can estimate it away and obtain the equality above as inequality $\leq$ without that term. Similarly, since $g_\eps \geq g$, the very last term is negative and we can omit that as well. Now by denoting $M:= \|g\|_\infty$, in total we have
\begin{align*}
  &\iint_{\Omega_T} (u_\eps - u)(u_\eps^m - u^m)\, \d x \d t\\
  &\qquad\leq \eps^m \iint_{\Omega_T} (u_\eps - u)\, \d x \d t 
 + \int_\Omega (g_\eps(x,0) - g(x,0)) \int_0^T (u_\eps^m - u^m ) \, \d s \, \d x \\
&\qquad\leq \eps^m C(m,M) |\Omega_T| + C(m,M)T \int_\Omega (g_\eps(x,0) - g(x,0))\, \d x,
\end{align*}
since the maximum principle, Lemma~\ref{l.maximum-principle}, implies $u \leq M$ and $u_\eps \leq (M^m + 1)^\frac{1}{m}$ a.e. in $\Omega_T$. Now we have that
$$
g_\eps(x,0) - g(x,0) = \left( g^m(x,0) + \eps^m\right)^\frac{1}{m} - g(x,0) \xrightarrow{\eps \to 0} 0
$$
pointwise a.e.~in $\Omega$. Also, $0 \leq g_\eps(x,0) - g(x,0) \leq
(2^\frac{1-m}{m}-1) g(x,0) + 2^\frac{1-m}{m} \in L^1(\Omega)$, such
that the dominated convergence theorem implies 
$$
\lim_{\eps \to 0} \int_\Omega (g_\eps(x,0) - g(x,0))\, \d x = 0.
$$
By choosing $\delta (\eps) = \max \left\{\eps^m, \int_\Omega (g_\eps(x,0) - g(x,0))\, \d x \right\}$, the claim follows.
\end{proof}

\subsection{Continuous weak solutions}

As an auxiliary tool, we will also use a local notion of continuous very weak solution, see~\cite{Abdulla1,Abdulla2}.

\begin{defin} \label{d.very-weak-cont}
We say that $u \in C(\overline{\Omega_T})$ is a continuous very weak
solution with boundary values $g \in C(\overline{\partial_p
  \Omega_T})$, if $u = g$ on $\overline{\partial_p \Omega_T}$ and for
every $0<t_1<t_2 \leq T$ and smooth $Q \Subset \Omega$
\begin{align*} 
\iint_{Q_{t_1,t_2}} &-(u \partial_t \eta + u^m \Delta \eta) \, \d x\d t + \int_{t_1}^{t_2} \int_{\partial Q} u^m \partial_\nu \eta \, \d \sigma \d t  \nonumber \\
&= \int_{Q} u(x,t_1) \eta(x,t_1) \, \d x - \int_{Q} u(x,t_2) \eta (x,t_2) \, \d x
\end{align*}
holds true for all $\eta \in C^{2,1}(\overline{Q_{t_1,t_2}})$ vanishing on $\partial Q \times (t_1,t_2]$, where is $\nu$ is the outward-directed normal vector to $Q$ at points on $\partial Q$.
\end{defin}

We recall existence, comparison and stability results for the notion defined above from~\cite{Abdulla1,Abdulla2}.

\begin{theo} \label{t.existence-abdulla}
Let $0<m<1$ and $\Omega_T$ be a $C^{1,\alpha}$-cylinder with $\alpha>0$. Then, for any $g \in C(\overline{\partial_p \Omega_T})$ there exists a unique locally H\"older continuous very weak solution $u\in C(\overline{\Omega_T})$ in the sense of Definition~\ref{d.very-weak-cont} such that $u = g$ on $\overline{\partial_p \Omega_T}$. Furthermore, if $u_1$ and $u_2$ are weak solutions with boundary values $g_1$ and $g_2$, respectively, satisfying $g_1 \leq g_2$, then $u_1 \leq u_2$.
\end{theo}

\begin{theo}[{\hspace{1sp}\cite[Corollary 2.3]{Abdulla2}}] \label{t.stability_existence}
Let $0 < m < 1$ and let $\Omega_T$ be a $C^{1,\alpha}$-cylinder
with $\alpha>0$. Also, let $h_j \in C(\overline{\partial_p \Omega_T})$ be
nonnegative, and let $u_j \in C(\overline{\Omega_T})$ be the corresponding
very weak solution given by Theorem~\ref{t.existence-abdulla}, for
$j\in\N_0$. If we have $\sup_{\overline{\partial_p \Omega}_T} \left| h_j - h_0 \right| \to 0$ as $j \to \infty$, then $\lim_{j\to \infty} u_j = u_0$ in $\overline{\Omega_T}$, and the convergence is locally uniform in $\Omega \times (0,T]$ as $j \to \infty$.
\end{theo}

Then we are at the stage of stating a useful result concerning existence and comparison of continuous weak solutions.

\begin{theo} \label{t.existence_continuous_weaksol}
Let $0<m<1$ and $\Omega_T$ be a $C^{1,\alpha}$-cylinder
with $\alpha >0$. Suppose that the function $g \in C(\overline{\Omega_T})$ satisfies $g^m \in L^2(0,T;H^1(\Omega))$ and $\partial_t g^m \in L^\frac{m+1}{m}(\Omega_T)$. Then, there exists a unique global weak solution $u$ to~\eqref{evo_eqn} such that $u \in C(\overline{\Omega_T})$, $u$ is locally H\"older continuous and $u = g$ on $\partial_p \Omega_T$. Moreover, if $g'$ satisfies conditions above, $g\leq g'$ on $\partial_p \Omega_T$ and $h' \in C(\overline {\Omega_T})$ is a global weak solution with boundary values $g'$ on $\partial_p \Omega_T$, then $h \leq h'$ in $\Omega_T$.
\end{theo}

\begin{proof}

By~\cite[Theorem 1.2]{S-existence} there exists a global weak solution
$u$ to~\eqref{evo_eqn} such that $u \in L^\infty(0,T;L^{m+1}(\Omega))$
and $u^m \in L^2(0,T;H^1(\Omega))$, and $u$ attains the lateral boundary values
in the sense $u^m - g^m \in L^2(0,T;H^1_0(\Omega))$ and the initial
values $g_o = g(x,0)$ in $L^{m+1}$-sense. 
Observe that since $g \in L^\infty(\Omega_T)$, also $u \in
L^\infty (\Omega_T)$ by the maximum principle, Lemma~\ref{l.maximum-principle}. Now~\cite[Theorem 18.1, Chapter 6]{DGV} implies that $u$ is locally H\"older continuous and~\cite{MSb} that $u \in C(\overline{\Omega_T})$. Furthermore, the solution is unique by~\cite[Theorem 5.3]{Vazquez}. It is a straightforward consequence that $u$ is a very weak solution
according to Definition~\ref{d.very-weak-cont} with boundary values $g$. 

By Theorem~\ref{t.existence-abdulla} there exists a unique locally H\"older continuous very weak solution
$\tilde u \in C(\overline{\Omega_T})$ according to
Definition~\ref{d.very-weak-cont} such that $\tilde u = g$ on
$\overline{\partial_p \Omega_T}$. By uniqueness $u$ and $\tilde u$
coincide. The comparison principle holds by Theorem~\ref{t.existence-abdulla}.

\end{proof}

\subsection{Some properties of weak supersolutions}

Next we state a Caccioppoli inequality for bounded weak supersolutions, see~\cite[Lemma 2.15]{KinnunenLindqvist_crelle}.

\begin{lem} \label{l.bounded_caccioppoli}
Let $m>0$. Suppose that $u \leq M$ is a weak supersolution in $\Omega_T$. Then, there exists a numerical constant $C > 0$ such that
$$
\int_{t_1}^{t_2}\int_{\Omega} \xi^2 \left| \nabla u^m  \right|^2 \, \d x \d t \leq  C M^{2m} T \int_\Omega \left| \nabla \xi \right|^2 \, \d x + C M^{m+1} \int_\Omega \xi^2 \, \d x
$$
for every $\xi = \xi(x) \in C_0^\infty (\Omega)$ with $\xi \geq 0$, and any $t_1,t_2$ satisfying $0<t_1<t_2<T$.
\end{lem}

In the following, for $v \in L^1_{\loc}(\Omega_T)$, $h >0$ and $\tau_1 > 0$, we use the mollification in time defined as
\begin{equation} \label{e.time-mollif}
\mollifytime{u}{h}(x,t) = \tfrac{1}{h} \int_{\tau_1}^t e^\frac{s-t}{h} u(x,s) \, \d s
\end{equation}
for any $t \in (\tau_1,T)$. For the standard properties of this mollification, see e.g.~\cite[Lemma 2.2]{KinnunenLindqvist2006}.

The proof of the next lemma follows the lines of~\cite[Lemma A.1]{BDL}, see also~\cite[Lemma 2.7]{lukkari-fde-measuredata}.

\begin{lem} \label{l.supersol_truncation}
Let $m>0$. If $u$ is a weak supersolution in $\Omega_T$, then $\min \{u,k\}$ is a weak supersolution in $\Omega_T$ for every $k \geq 0$.
\end{lem}

\begin{proof}
Let us start with a mollified weak formulation
$$
\iint_{\Omega_T} \partial_t \mollifytime{u}{h} \varphi + \mollifytime{\nabla u^m}{h} \cdot \nabla \varphi \, \d x \d t \geq 0
$$
for $\varphi\in C^\infty_0(\Omega_T,\R_{\ge0})$, 
and use a test function $\varphi = \eta \frac{(u^m - k^m)_-}{(u^m-k^m)_- + \sigma}$ with $\sigma > 0$ and $\eta \in C_0^\infty(\Omega_T,\R_{\geq 0})$.
For the divergence part we have
\begin{align*}
\lim_{h\to 0}&\iint_{\Omega_T}  \mollifytime{\nabla u^m}{h} \cdot \nabla \varphi \, \d x \d t \\
&= \iint_{\Omega_T}  \nabla u^m \cdot \left(\nabla \eta \frac{(u^m - k^m)_-}{(u^m-k^m)_- + \sigma} + \sigma \eta \frac{ \nabla (u^m - k^m)_-}{\left[(u^m-k^m)_- + \sigma\right]^2} \right) \, \d x \d t \\
&\leq \iint_{\Omega_T}  \nabla u^m \cdot \nabla \eta \frac{(u^m - k^m)_-}{(u^m-k^m)_- + \sigma} \, \d x \d t \\
&\longrightarrow \iint_{\Omega_T} \nabla (\min\{u,k\}^m) \cdot \nabla \eta \, \d x \d t 
\end{align*}
as $\sigma \to 0$ by the dominated convergence theorem. For the parabolic part we obtain
\begin{align*}
\iint_{\Omega_T} \partial_t \mollifytime{u}{h} \varphi\, \d x \d t &= \iint_{\Omega_T} \eta \partial_t \mollifytime{u}{h} \frac{(\mollifytime{u}{h}^m - k^m)_-}{(\mollifytime{u}{h}^m-k^m)_- + \sigma} \, \d x \d t \\
&\phantom{+} + \iint_{\Omega_T} \eta \partial_t \mollifytime{u}{h} \left(\frac{(u^m - k^m)_-}{(u^m-k^m)_- + \sigma} - \frac{(\mollifytime{u}{h}^m - k^m)_-}{(\mollifytime{u}{h}^m-k^m)_- + \sigma} \right) \, \d x \d t \\
&\leq \iint_{\Omega_T} \eta \partial_t \mollifytime{u}{h} \frac{(\mollifytime{u}{h}^m - k^m)_-}{(\mollifytime{u}{h}^m-k^m)_- + \sigma}\, \d x \d t,
\end{align*}
since the map $s \mapsto \frac{(s^m-k^m)_-}{(s^m - k^m)_- + \sigma}$
is decreasing and $\partial_t \mollifytime{u}{h}=\frac1h(u-\mollifytime{u}{h})$. Now we can estimate further
\begin{align*}
\iint_{\Omega_T} &\eta \partial_t \mollifytime{u}{h} \frac{(\mollifytime{u}{h}^m - k^m)_-}{(\mollifytime{u}{h}^m-k^m)_- + \sigma}\, \d x \d t \\
&= \iint_{\Omega_T} \eta \partial_t \left[k - \int_{\mollifytime{u}{h}}^k \frac{(s^m - k^m)_-}{(s^m-k^m)_- + \sigma} \, \d s \right]\, \d x \d t \\
&= - \iint_{\Omega_T} \partial_t \eta \left[k - \int_{\mollifytime{u}{h}}^k \frac{(s^m - k^m)_-}{(s^m-k^m)_- + \sigma} \, \d s \right]\, \d x \d t \\
&\xrightarrow[]{h\to 0}- \iint_{\Omega_T} \partial_t \eta \left[k - \int_{u}^k \frac{(s^m - k^m)_-}{(s^m-k^m)_- + \sigma} \, \d s \right]\, \d x \d t \\
&\xrightarrow{\sigma \to 0} - \iint_{\Omega_T} \partial_t \eta [k - (u-k)_-] \, \d x \d t.
\end{align*}
Since $k - (u-k)_- = \min \{u,k\}$, in total we have
$$
\iint_{\Omega_T} - \min\{u,k\} \partial_t \eta + \nabla (\min\{u,k\}^m) \cdot \nabla \eta \, \d x \d t \geq 0,
$$
which completes the proof.
\end{proof}

A result in~\cite{Naian} states that every weak supersolution has a lower semicontinuous representative.

\begin{theo} \label{t.super_lsc}
Let $m > 0$ and $u$ be a weak supersolution according to
Definition~\ref{d.weak_sol}. Then, there exists a lower semicontinuous
function $u_*$ such that $u_*(x,t) = u(x,t)$ for a.e. $(x,t) \in
\Omega_T$. Moreover, 
$$
u_*(x,t) = \essliminf_{\substack{(y,s) \to (x,t) \\ s<t}} u(y,s),
$$
for every $(x,t) \in \Omega_T$.
\end{theo}

\section{Notion of supercaloric functions}

Up next we define (quasi-)super- and subcaloric functions.

\begin{defin} \label{d.supercal}
  Let $U\subset\R^{n+1}$ be an open set.
  A function $u \colon U \to [0,\infty]$ is called a supercaloric function, if
\begin{itemize}
\item[(i)] $u$ is lower semicontinuous,
\item[(ii)] $u$ is finite in a dense subset,
\item[(iii)] $u$ satisfies the comparison principle in every subcylinder $Q_{t_1,t_2}=Q \times (t_1,t_2) \Subset U$: if $h\in C(\overline{Q}_{t_1,t_2})$ is a weak solution in $Q_{t_1,t_2}$ and if $h \leq u$ on the parabolic boundary of $Q_{t_1,t_2}$, then $ h\leq u$ in $Q_{t_1,t_2}$.
\end{itemize}
We call $u$ a quasi-supercaloric function if (i) and (ii) hold, and (iii) is replaced by
\begin{itemize}
\item[(iii')] $u$ satisfies the comparison principle in every $C^{2,\alpha}$-subcylinder $Q_{t_1,t_2}=Q \times (t_1,t_2) \Subset U$: if $h\in C(\overline{Q}_{t_1,t_2})$ is a weak solution in $Q_{t_1,t_2}$ and if $h \leq u$ on the parabolic boundary of $Q_{t_1,t_2}$, then $ h\leq u$ in $Q_{t_1,t_2}$.
\end{itemize}

A function $u: \Omega_T \to [0,\infty)$ is called subcaloric function
if the conditions (i), (ii) and (iii) above hold with (i) replaced by
upper semicontinuity, and inequalities in (iii) by $\geq$. The
function $u$ is called quasi-subcaloric if (iii') holds instead of (iii) with $\geq$.
\end{defin}

%
%

The notion of quasi-supercaloric functions is only used as an
auxiliary construct for the following proofs. In fact, it turns out
that the classes of supercaloric and quasi-supercaloric functions coincide, see
Proposition~\ref{p.supercaloric-general-comparison}.
However, the proof requires a more detailed analysis of
quasi-supercaloric functions and is therefore postponed to the end of
this section.

Our next goal is to prove that every lower semicontinuous weak
supersolution is a supercaloric function.  Observe that a weak
supersolution is lower semicontinuous after a possible redefinition in
a set of measure zero by Theorem~\ref{t.super_lsc}. However, since the comparison principle from
Lemma~\ref{l.comparison-weaksupersub} is limited to
$C^{2,\alpha}$-cylinders, as a first step we only obtain the following
preliminary result.

\begin{lem} \label{l.weasuper-is-quasi-supercal}
Let $0<m<1$. If $u$ is a weak supersolution in
$\Omega_T$, then $u_*$ is a quasi-supercaloric function in $\Omega_T$.
\end{lem}
\begin{rem}
  At the end of this  section we will improve this result and show
  that lower semicontinuous weak
  supersolutions are supercaloric functions, see
  Lemma~\ref{l.weasuper-is-supercal}.
\end{rem}

\begin{proof}
We only need to show the comparison principle (iii') from the
definition of quasi-supercaloric functions. Let
$Q_{t_1,t_2}\Subset\Omega_T$ be a $C^{2,\alpha}$-cylinder, and $h \in
C(\overline{Q_{t_1,t_2}})$ a weak solution, which implies $h^m \in
L^2_{\loc}(t_1,t_2;H^1_{\loc}(Q))$. We are not able to use the comparison
principle between weak subsolutions and supersolutions,
Lemma~\ref{l.comparison-weaksupersub}, directly, since we would need
$h^m \in L^2(t_1,t_2;H^1(Q))$. Thus we proceed as follows.

Denote $\tilde u = \min \{u_*,\max_{\overline{Q}_{t_1,t_2}} h\}$, which
is a lower semicontinuous weak supersolution by
Lemma~\ref{l.supersol_truncation}. We let $\bar h_j:\overline{Q_{t_1,t_2}}\to\R_{\ge0}$ be Lipschitz
functions for $j = 1,2,...$, such that for
$h_j := \bar h_j \big|_{\partial_p Q_{t_1,t_2}}$ we have $0 \leq h_j \leq h^m$ on $\partial_p Q_{t_1,t_2}$ and 
\begin{equation} \label{e.bvalues-uniform-conv}
\sup_{\partial_p Q_{t_1,t_2}} | h_j^\frac{1}{m} - h |  \xrightarrow{j\to \infty} 0.
\end{equation}
By Theorem~\ref{t.existence_continuous_weaksol} there exists a unique weak solution $\hat
h_j \in C(\overline{Q_{t_1,t_2}})$ in $Q_{t_1,t_2}$ taking the
boundary values $h_j^\frac{1}{m}$ continuously, and $\hat h_j^m -
\bar h_j \in L^2(t_1,t_2;H^1_0(Q))$. By
Lemma~\ref{l.comparison-weaksupersub}, we have that $\hat h_j(x,t) \leq \tilde u(x,t) \leq u_*(x,t)$ for a.e. $(x,t) \in Q_{t_1,t_2}$. Since $u=u_*$ a.e. by Theorem~\ref{t.super_lsc}, it follows that $(u_*)_* = u_*$ everywhere. Together with continuity of $\hat h_j$ it follows that $\hat h_j(x,t) \leq u_*(x,t)$ for every $(x,t)\in Q_{t_1,t_2}$. 

Furthermore since~\eqref{e.bvalues-uniform-conv} holds,
Theorem~\ref{t.stability_existence} implies that also in the limit
$j\to \infty$, $h(x,t)\leq u_*(x,t)$ holds for every $(x,t) \in
Q_{t_1,t_2}$. Thus $u_*$ is a
quasi-supercaloric function.

\end{proof}

In the next lemma we show that the comparison principle for super(sub)caloric functions holds in general space-time cylinders. The proof follows the lines of~\cite[Theorem 3.6]{Bjorns_boundary} (see also~\cite[Theorem 3.3]{KLL}), in which the result was proved in case $m \geq 1$. Observe that the result is proved for quasi-super(sub)caloric functions, which implies that the result also holds for super(sub)caloric functions. 

\begin{lem} \label{l.supersubcal-cylinder-comparison}
Suppose that $0<m<1$. Let $Q_{t_1,t_2} \Subset \R^{n+1}$ be a cylinder. Suppose that $u$ is a \textup{(}quasi\mbox{-\textup{)}}supercaloric and $v$ is a \textup{(}quasi\mbox{-\textup{)}}subcaloric function in $Q_{t_1,t_2}$. If
$$
\infty \neq \limsup_{Q_{t_1,t_2} \ni (y,s) \to (x,t)} v(y,s) \leq \liminf_{Q_{t_1,t_2} \ni (y,s) \to (x,t)} u(y,s)
$$
for every $(x,t) \in \partial_p Q_{t_1,t_2}$, then $v \leq u$ in $Q_{t_1,t_2}$.
\end{lem}

\begin{proof}
Fix $\delta > 0$ and denote $\tau_2 := t_2 - \delta$, $\tilde \tau_2 := t_2 - \frac{\delta}{2}$ and $\hat \tau_2 := t_2 - \frac{\delta}{4}$. If $u$ is unbounded, we may consider $\tilde u = \min \{u,
\sup_{Q_{t_1,\hat \tau_2}} v\}$ instead of $u$ in the proof, which is a
bounded quasi-supercaloric function in $Q_{t_1,\tilde \tau_2}$ as a truncation of a quasi-supercaloric
function. Then in the end, by proving $v\leq \tilde u$ in
$Q_{t_1, \tau_2}$ this implies $v  \leq u$ in $Q_{t_1, \tau_2}$ since $\tilde
u \leq u$ in $Q_{t_1, \tau_2}$. Therefore, from now on we assume that $u$ is bounded. Furthermore, observe that $v$ is locally bounded in $Q_{t_1,t_2}$ by definition, and the assumption implies that $v$ is bounded in $Q_{t_1,\hat \tau_2}$.

We extend $u$ up to the parabolic boundary by setting 
$$
u(x,t) := \liminf_{Q_{t_1,t_2} \ni (y,s) \to (x,t)} u(y,s)\quad \text{ for every } (x,t) \in \overline{\partial_p Q_{t_1,\tilde \tau_2}}.
$$
The function $v$ is extended analogously via $\limsup$. By standard arguments it follows that $u$ ($v$) is lower(upper) semicontinuous in $\overline{Q_{t_1,\tilde \tau_2}}$.

For $\eps_j = 1/j$, take nested $C^{2,\alpha}$-cylinders
$Q^j_{s_j,\tilde \tau_2}  \Subset Q \times (t_1,\tilde \tau_2]$
with 
$$
\bigcup_{j=1}^{\infty} Q^j = Q,\quad s_j \xrightarrow{j\to \infty} t_1
$$
and
$$
v^m \leq u^m + \tfrac12\eps_j^m \quad \text{ in }  \overline{Q_{t_1,\tilde \tau_2}} \setminus \left( Q^j \times (s_j, \tilde \tau_2] \right).
$$
We can find a non-decreasing sequence of functions $\bar h_j \in C^{0,1}(\overline{Q_{t_1,\tau_2}} ,\R_{\ge0})$ such that $\bar h_j \xrightarrow{j\to \infty} u^m$ pointwise in $\overline{Q_{t_1,\tau_2}}$ satisfying
$$
v^m \leq \bar h_j + \eps_j^m \leq u^m + \eps_j^m \quad \text{ in } \overline{Q_{t_1,\tau_2}} \setminus \left( Q^j \times (s_j,\tau_2] \right).
$$
Observe that by construction $\|\bar h_j\|_{L^\infty(Q_{t_1,\tau_2})} \leq \|u^m\|_{L^\infty(Q_{t_1,\tau_2})} < \infty$ for every $j \in \N$.

In view of Theorem~\ref{t.existence_continuous_weaksol}, we can find continuous global weak solutions $h_j$ and $\hat h_j$ in $Q^j_{s_j,\tau_2}$
that take the boundary values $\bar h_j^\frac{1}{m}$ and $(\bar h_j +
\eps_j^m)^\frac{1}{m}$ continuously and in the Sobolev/trace sense on
$\partial_p Q^j_{s_j,\tau_2}$. Since $v$ is quasi-sub- and $u$
quasi-supercaloric, and $Q^j_{s_j,\tau_2}\Subset Q_{t_1,t_2}$ are
$C^{2,\alpha}$-cylinders, we have that 
$$
u \geq h_j\quad \text{ and } \quad v \leq \hat h_j \quad \text{ in } Q^j_{s_j,\tau_2}.
$$
By extending $h_j$ by $\bar h_j^\frac{1}{m}$ and $\hat h_j$ by $(\bar h_j + \eps_j^m)^\frac{1}{m}$ to $Q_{t_1,\tau_2} \setminus Q^j_{s_j,\tau_2}$, the inequalities above hold also in this set. Furthermore, we clearly have 
$$
h_j \leq \hat h_j  \quad \text{ in } Q_{t_1,\tau_2} \setminus Q^j_{s_j,\tau_2},
$$
and
$$
h_j \leq \hat h_j \quad \text{ in } Q^j_{s_j,\tau_2}
$$
by the comparison principle for weak solutions, see
Lemma~\ref{l.comparison-weaksupersub}. Furthermore, sequences of
functions $h_j$ and $\hat h_j$ are uniformly bounded in $Q_{t_1,
  \tau_2}$ since $\bar h_j$ is by the maximum principle from Lemma~\ref{l.maximum-principle}.

By the estimate for the local H\"older continuity~\cite[Theorem 18.1 ,
Chapter 6]{DGV} we have that the families $h_j$ and $\hat h_j$ are locally
equicontinuous, which by Arzel\`a-Ascoli and a diagonal argument shows that there exist subsequences $h_j$ and $\hat h_j$ that converge locally uniformly in $Q_{t_1,\tau_2}$ to continuous functions $h$ and $\hat h$, which satisfy $h \leq \hat h$, and by earlier inequalities also 
\begin{equation} \label{e.uh-vhath}
u \geq h \quad \text{ and } \quad v \leq \hat h  \quad \text{ in } Q_{t_1,\tau_2}.
\end{equation}
Let us restrict to a subsequence for which the aforementioned convergences hold. By using~\cite[Corollary 3.11]{BLS} and Lemma~\ref{l.u_eps-u-estimate}, we have
\begin{align*}
&\iint_{Q_{t_1,\tau_2}} |\hat h_j^m - h_j^m|^\frac{m+1}{m} \, \d x \d t  \\
&\leq \iint_{Q^j_{s_j,\tau_2}} (\hat h_j - h_j) (\hat h_j^m - h_j^m) \, \d x \d t \\
&\phantom{+}+ \iint_{Q_{t_1,\tau_2}\setminus Q^j_{s_j,\tau_2}} |\hat h_j^m - h_j^m|^\frac{m+1}{m} \, \d x \d t\\
&\leq c(m,\|\bar h_j\|_\infty,|Q|,t_2-t_1) \max \left\{\eps_j^m, \int_{Q^j} \left( \bar h_j(x,s_j) + \eps_j^m \right)^\frac{1}{m}- \bar h_j(x,s_j)^\frac{1}{m} \, \d x \right\} \\
&\phantom{+} + \eps_j^{m+1} |Q_{t_1,\tau_2}\setminus Q^j_{s_j,\tau_2}| \\
&\leq c(m,\|u\|_\infty,|Q|,t_2-t_1) \max \left\{\eps_j^m, \left( \|u\|^m_\infty + \eps_j^m \right)^\frac{1}{m}- \|u\|_\infty \right\} \\
&\xrightarrow{j \to \infty} 0,
\end{align*}
where we used the facts $\|\bar h_j\|_\infty \leq \|u\|_\infty^m < \infty$ and $s \mapsto (s+\eps_j^m)^\frac{1}{m} - s^\frac{1}{m}$ is a non-decreasing mapping.

Since the functions $h_j, \hat h_j$ are uniformly bounded in $Q_{t_1, \tau_2}$, $h_j \to h$
and $\hat h_j \to \hat h$ pointwise in $Q_{t_1,\tau_2}$, the estimate above together with the dominated convergence theorem implies
$$
\iint_{Q_{t_1,\tau_2}} |\hat h^m - h^m|^\frac{m+1}{m} \, \d x \d t \leq 0.
$$
Thus $\hat h = h$ a.e. in $Q \times (t_1,\tau_2)$. By continuity of $\hat h$ and $ h$ this holds at every point, which together with~\eqref{e.uh-vhath} concludes the result in $Q \times (t_1,\tau_2) = Q \times (t_1,t_2-\delta)$. Since $\delta > 0$ was arbitrary, the result holds in $Q \times (t_1,t_2)$.

\end{proof}

For the proof of the following two lemmas in the case $m \geq 1$, see~\cite[Proposition 3.8 and Theorem 3.5]{Bjorns_boundary}.

\begin{prop} \label{p.supercaloric-general-comparison}
Let $0<m<1$. If $u$ is a quasi-supercaloric function, then $u$ is a supercaloric function.
\end{prop}

\begin{proof}
  Let $Q_{t_1,t_2} \Subset \Omega_T$ and
  $h \in C(\overline{Q_{t_1,t_2}})$ be a weak solution in
  $Q_{t_1,t_2}$ such that $h \leq u$ on $\partial_p
  Q_{t_1,t_2}$. Since $h$ is continuous in $\overline{Q_{t_1,t_2}}$,
  it is also bounded in $\overline{Q_{t_1,t_2}}$. By an analogous
  proof as in Lemma~\ref{l.weasuper-is-quasi-supercal}, 
  $h$ is a
  quasi-subcaloric function. Since $u$ 
  is a
  quasi-supercaloric function, we may use
  Lemma~\ref{l.supersubcal-cylinder-comparison}
  to conclude that $h \leq u$ in $Q_{t_1,t_2}$, which implies the claim.
\end{proof}

Combining the preceding proposition with
Lemma~\ref{l.weasuper-is-quasi-supercal}, we obtain  the desired
improvement of Lemma~\ref{l.weasuper-is-quasi-supercal}.  

\begin{lem} \label{l.weasuper-is-supercal}
Let $0<m<1$ and $u$ be a weak supersolution in $\Omega_T$. Then, $u_*$ is a supercaloric function in $\Omega_T$.
\end{lem}

In the next lemma we show that supercaloric functions can be extended by zero in the past.

\begin{lem} \label{l.zero-past-extension}
Let $0<m<1$ and $v : \Omega_T \to [0,\infty]$ be a supercaloric function in $\Omega_T$. Then
\[ u =
\begin{cases}
v\quad &\text {in } \Omega \times (0,T), \\
0\quad &\text {in } \Omega \times (-\infty,0],
\end{cases}
\]
is a supercaloric function in $\Omega \times (-\infty,T)$.
\end{lem}

\begin{proof}
Clearly $u$ satisfies items (i) and (ii) in Definition~\ref{d.supercal} since $v$ does, and $v \geq 0$. By showing (iii'), the claim holds by Proposition~\ref{p.supercaloric-general-comparison}.

Fix a $C^{2,\alpha}$-cylinder $Q_{t_1,t_2} \Subset \Omega \times (-\infty,T)$, and let $h \in C(\overline{Q_{t_1,t_2}})$ be a weak solution in $Q_{t_1,t_2}$ such that $h \leq u$ on $\partial_p Q_{t_1,t_2}$. Furthermore, suppose that $Q_{t_1,t_2} \cap (\Omega \times \{0\} ) \neq \varnothing$ since otherwise the comparison (in (iii') of Definition~\ref{d.supercal}) clearly holds.

By definition of $u$ we have that $h \leq v = 0$ on $\partial_p [Q \times (t_1,0)]$, i.e. $h = 0$ on $\partial_p [Q \times (t_1,0)]$. This implies that $h =0 $ in $Q \times (t_1,0)$. Since $h$ is continuous, this implies that $h = 0$ in $Q \times (t_1,0]$. Now by using also continuity of $h$ we have that 
$$
\limsup_{Q_{0,t_2} \ni (y,s) \to (x,t)} h(y,s) = h(x,t) \leq \liminf_{Q_{0,t_2} \ni (y,s) \to (x,t)} v(y,s)
$$
for all $(x,t) \in \partial_p Q_{0,t_2}$. Since $v$ is supercaloric and $h$ is subcaloric in $Q_{0,t_2}$, it follows that $h \leq v$ in $Q_{0,t_2}$ by Lemma~\ref{l.supersubcal-cylinder-comparison} completing the proof.
\end{proof}

Then we recall a parabolic comparison principle for super(sub)caloric functions in noncylindrical bounded sets from~\cite[Theorem 5.1]{Bjorns_boundary}.

\begin{lem} \label{l.comparison_t<T}
Let $m> 0$ and $U \subset \R^{n+1}$ be a bounded open set. Suppose
that $u$ is a supercaloric and $v$ is a subcaloric function in $U$. Let $T \in \R$ and assume that
$$
 \limsup_{U \ni (y,s)\to (x,t)} v(y,s) < \liminf_{U \ni (y,s) \to (x,t)} u(y,s)
$$
for all $(x,t) \in \{  (x,t) \in \partial U : t < T  \}$. Then $v \leq u$ in $\{  (x,t)\in U : t < T  \}$.
\end{lem}

The following result shows that the class of supercaloric functions is closed under increasing limits, provided that the limit function is finite in a dense set~\cite[Proposition 4.6]{Bjorns_boundary}. 

\begin{lem} \label{l.superc_increasing_lim}
Let $m>0$ and $u_k$ be a nondecreasing sequence of supercaloric functions in $\Omega_T$. If $u := \lim_{k\to \infty} u_k$ is finite in a dense subset of $\Omega_T$, then $u$ is a supercaloric function in $\Omega_T$.
\end{lem}

\section{Positivity sets of supercaloric functions} \label{sec:positivity}

First we recall the following result on
expansion of positivity for weak solutions.

\begin{theo}[{\hspace{1sp}\cite[Chapter 4, Prop. 7.2]{DGV}}] \label{t.expansionofpositivity}
Let $0<m<1$. Assume that $u$ is a locally bounded, nonnegative weak solution
to \eqref{evo_eqn}
in class $C_{\loc}(0,T; L^{m+1}_{\loc}(\Omega))$. Suppose that for some $(x_o,t_o)\in \Omega_T$ and $r > 0$
$$
\left| \{ u(\cdot,t_o) \geq M \} \cap B(x_o,r) \right| \geq \alpha |B(x_o,r)|
$$
holds true for some $M > 0$ and $\alpha \in (0,1)$. Then there exist constants $\eps, \delta, \eta \in (0,1)$ depending only on $n$, $m$ and $\alpha$	such that 
$$
u(\cdot,t) \geq \eta M\quad \text{in } B(x_o,2r)
$$ 
for all
$$
t \in \left[ t_o + (1 - \eps) \delta M^{1-m} r^2 ,  t_o + \delta M^{1-m} r^2  \right],
$$
provided that $B (x_o, 16 r) \times (t_o, t_o + \delta M^{1-m} r^2) \Subset \Omega_T$.

\end{theo}

We use the expansion of positivity for the following characterization of
the positivity set of supercaloric functions in the 
fast diffusion case. 

\begin{lem}\label{lem:alternatives}
  Let $0<m<1$ 
  and assume that $u$ is a non-negative
  supercaloric function in $\Omega_T$,  where $\Omega\subset\R^n$ is open and
  connected. Then, for any time $t\in(0,T)$ either
  $u$ is positive on the whole time slice $\Omega\times\{t\}$ or $u$
  vanishes on the whole time slice.
\end{lem}

\begin{proof}
  As a first step, we prove the claim for a continuous, non-negative,
  bounded weak solution to~\eqref{evo_eqn}. Let us fix a time
  $t\in(0,T)$.
  We claim that $u_o:=u(x_o,t)>0$ for some $x_o\in\Omega$
  implies 
  \begin{equation}\label{positivity-ball}
    u(\cdot,t)>0 \qquad\mbox{in $B(x_o,r)$},
  \end{equation}
  for any $r>0$ with $B(x_o, 16r)\Subset\Omega$. First, we note that
  the continuity of $u$ implies
  \begin{equation}\label{pos-small-cylinder}
    u\ge \tfrac12 u_o
    \mbox{\qquad in }B(x_o,\rho)\times[t-\rho^2,t]\subset\Omega_T
  \end{equation}
  for some $\rho>0$. If $r\le\rho$, this already implies
  claim~\eqref{positivity-ball}. Otherwise, we apply
  Theorem~\ref{t.expansionofpositivity} with the parameter
  $\alpha:=(\frac\rho r)^n\in(0,1)$. Let 
  $\delta=\delta(n,m,\alpha)\in(0,1)$ be the number determined by this
  theorem. We choose $M\in (0,\tfrac12u_o]$ so small that
  \begin{equation*}
    \delta M^{1-m}r^2\le\rho^2
  \end{equation*}
  and let $t_o:= t-\delta M^{1-m}r^2\in[t-\rho^2,t]$. Because
  of~\eqref{pos-small-cylinder} and $M\le\tfrac12u_o$, we have
  \begin{equation*}
    \left| \{ u(\cdot,t_o) \geq M \} \cap B(x_o,r) \right|
    \geq |B(x_o,\rho)|
    = \alpha |B(x_o,r)|.
  \end{equation*}
  Therefore, Theorem~\ref{t.expansionofpositivity} implies
  \begin{equation*}
    u(\cdot,t)\ge\eta M\qquad\mbox{in }B(x_o,2r)
  \end{equation*}
  for some $\eta>0$, which implies claim~\eqref{positivity-ball}.
  Next, we observe that this yields the implication 
  \begin{equation}\label{implication}
    u(x_o,t)>0\mbox{\ in some point $x_o\in\Omega$}\implies
    u(x_1,t)>0\mbox{\ in any point $x_1\in\Omega$.}
  \end{equation}
  For the derivation of this claim, we recall that $\Omega$ is
  connected and consider a curve
  $\Gamma\subset\Omega$ that connects
  $x_o$ and $x_1$. Then we cover $\Gamma$ by finitely many balls
  $B(x_i,r)$, $i=1,\ldots,L$, with $x_{i+1}\in B(x_i,r)$ for any
  $i=0,\ldots,L-1$. Since $\Gamma$ is compactly contained in
  $\Omega$, we can choose the radius $r>0$
  small enough to ensure $B(x_i,16r)\Subset\Omega$ for each
  $i=1,\ldots,L$.
  Repeated applications of the positivity
  result~\eqref{positivity-ball}
  imply that $u$ is positive on each of
  the balls $B(x_i, r)\times\{t\}$, and in particular $u(x_1,t)>0$.
  
   This proves claim~\eqref{implication}. The contraposition of
   this implication ensures that $u(x_1,t)=0$ for some $x_1\in\Omega$ implies
   $u(x_o,t)=0$ in any point $x_o\in\Omega$. We conclude that either
   $u$ is positive or zero on the whole time slice $\Omega\times\{t\}$.
   This proves the 
   claim for a continuous, bounded weak solution.

   Now, we consider a supercaloric function $u:\Omega_T\to[0,\infty]$.
   Let us assume for contradiction that there is a time $t\in(0,T)$
   for which $\Omega_+:=\{x\in\Omega\colon u(x,t)>0\}$ satisfies
   $\varnothing\neq\Omega_+\subsetneq\Omega$. By lower semicontinuity of $u$,
   the set $\Omega_+$ is open. Because $\Omega$ is connected,
   its subset $\Omega_+$ can not be relatively closed. Therefore,
   there exists a point $x_o\in\partial\Omega_+\cap\Omega$, in which
   we have $u(x_o,t)=0$. We choose a neighborhood
   $B(x_o,r)\subset\Omega$. Because of $x_o\in\partial\Omega_+$, there
   exists a point $x_+\in B(x_o,r)$ with $u(x_+,t)>0$. If $u(x_+,t)< \infty$, let $a := u(x_+,t)$. If $u(x_+,t)= \infty$, let $a \in \R_{>0}$. By lower
   semicontinuity of $u$, there exists a $\delta>0$ such that
   \begin{equation*}
     u>\tfrac12a\mbox{\qquad on $B(x_+,2\delta)\times[t-\delta,t+\delta]\Subset\Omega_T$.}
   \end{equation*}
   We choose a function $\eta\in C^\infty_0(B(x_+,2\delta),[0,1])$ with
   $\eta\equiv1$ in $B(x_+,\delta)
   $ and abbreviate
   $\rho:=|x_o-x_+|$. Then we consider the weak
   solution to the Cauchy-Dirichlet problem
   \begin{equation*}
     \left\{
       \begin{array}{cl}
         \partial_tv-\Delta v^m=0&\mbox{in }B(x_o,\rho)\times(t-\delta,t+\delta),\\
         v=\tfrac12a\eta^\frac{1}{m}&\mbox{on }\partial_p[B(x_o,\rho)\times(t-\delta,t+\delta)].
       \end{array}
     \right.  
   \end{equation*}
   Theorem~\ref{t.existence_continuous_weaksol} implies that $v$ is non-negative,
bounded and continuous up to the boundary. 
Therefore, the first part of the proof implies that for every time
$s\in(t-\delta,t+\delta)$, the function $v$ is either positive on the
whole time slice $B(x_o,\rho)\times\{s\}$ or it vanishes on the whole
time slice. However, since $(x_+,s)\in
\partial_p[B(x_o,\rho)\times(t-\delta,t+\delta)]$ for every
$s\in(t-\delta,t+\delta)$, and in this point we have
$v(x_+,s)=\tfrac12a>0$, we can exclude the second alternative. This
proves $v>0$ on the whole domain
$B(x_o,\rho)\times(t-\delta,t+\delta)$.

Moreover, by construction we have $u\ge v$ on
$\partial_p[B(x_o,\rho)\times(t-\delta,t+\delta)]$. Therefore, by
definition of the supercaloric function $u$ we have $u\ge v$ on
$B(x_o,\rho)\times(t-\delta,t+\delta)$, and in particular
\begin{equation*}
  u(x_o,t)\ge v(x_o,t)>0.
\end{equation*}
Since $u(x_o,t)=0$ by construction, this yields the desired
contradiction. Therefore,  we have established the claim also
in the case of a supercaloric function.  
\end{proof}

\begin{cor} \label{c.positivity-intervals}
  Let $0<m<1$ 
  and assume that $u$ is a non-negative
  supercaloric function in $\Omega_T$,  where $\Omega\subset\R^n$ is open and
  connected. Then, the set 
    \begin{equation}\label{def:Lambda+}
    \Lambda_+:=\big\{t\in(0,T)\colon u\mbox{\ is positive on }\Omega\times\{t\}\big\}
  \end{equation}
  can be written as a countable union $\Lambda_+ = \bigcup_i \Lambda_i$, where $\Lambda_i$ is an open subinterval of $(0,T)$ for every $i$.
\end{cor}

\begin{proof}
In view of Lemma \ref{lem:alternatives} and since $u$ is lower semicontinuous, the set 
  \begin{equation*}
    \Lambda_+:=\big\{t\in(0,T)\colon u\mbox{\ is positive on }\Omega\times\{t\}\big\}
  \end{equation*}
  is an open subset of $(0,T)$. We decompose $\Lambda_+$ in its
  connected components $\Lambda_i=(t_{i,1},t_{i,2})$, $i\in I$ ,
  i.e. $\Lambda_+=\bigcup_{i\in I}\Lambda_i$, with disjoint open intervals
  $\Lambda_i$. Since $\Lambda_+$ is an open subset of the real line, there can be at most countably many connected components, i.e. we
  can choose the index set either as $I=\N$ or of the form
  $I=\{1,\ldots,L\}$.

\end{proof}

We state Harnack type estimates for weak solutions that will be used later on. In the following, we denote $\lambda := n(m-1) + 2$.

\begin{lem}[{\hspace{1sp}\cite[Chapter 6, Thm. 17.1]{DGV}}] \label{l.l1linfty_harnack}
Let $\frac{n-2}{n}<m<1$. Suppose that $u$ is a nonnegative weak solution in class $C_{\loc}(0,T; L^{m+1}_{\loc}(\Omega))$. Then there exists $\gamma = \gamma(n,m)$ such that 
$$
\sup_{B(y,r)\times [s,t]} u \leq  \frac{\gamma}{(t-s)^\frac{n}{\lambda}} \left(  \inf_{2s - t < \tau < t} \int_{B(y,2r)} u(x,\tau)\, \d x     \right)^\frac{2}{\lambda} + \gamma \left( \frac{t-s}{r^2}  \right)^\frac{1}{1-m}
$$
for all cylinders $B(y,2r) \times \left[ s- (t-s), s + (t-s)  \right] \Subset \Omega_T$.
\end{lem}

\begin{lem}[{\hspace{1sp}\cite[Prop. B.1.1]{DGV}}] \label{l.l1harnack}
Let $0<m<1$. Suppose that $u$ is a continuous nonnegative weak solution in $\Omega_T$.
Then there exists $\gamma = \gamma(n,m)\ge1$ such that 
$$
\sup_{s<\tau<t} \int_{B(y,r)} u(x,\tau)\, \d x \leq \gamma \inf_{s<\tau<t} \int_{B(y,2r)} u(x,\tau)\, \d x + \gamma \left( \frac{t-s}{r^\lambda} \right)^\frac{1}{1-m}
$$
for all cylinders $B(y,2r) \times [s,t] \Subset \Omega_T$. 
\end{lem}

Up next we prove a weak Harnack inequality for supercaloric functions. The proof follows the approach in~\cite[Proposition 3.1]{GLL}.

\begin{lem} \label{l.wharnack_supercal}
Let $\frac{n-2}{n}< m < 1$ and $u$ be a supercaloric function in
$\Omega_T$. Then, there exist constants $c_1,c_2, \alpha \in (0,1)$
depending only on $n$ and $m$, such that the following holds. Assume
that for some  $s\in (0,T)$, we have 
$$
\theta := c_2 \left( \bint_{B(x_o,2r)} u(x,s)\, \d x \right)^{1-m}
>0,
$$
and $B(x_o,64r) \times \left( s, s+ \theta r^2  \right) \Subset \Omega_T$.
Then the estimate 
$$
\inf_{B(x_o,2r)} u(\cdot,t) \geq c_1 \bint_{B(x_o,2r)} u(x,s) \, \d x
$$
holds for any $t\in \left[ s + \alpha \theta r^2 , s + \theta r^2 \right]$.
\end{lem}

\begin{proof}
  
Let us assume $(x_0,s) = (0,0)$ and $Q_S := B_{64r} \times (0,S)\Subset\Omega_T$, for some $S <T$.
Let $u$ be a supercaloric function in $\Omega_T$, and $u_k:=\min\{u,k\}$ its
truncation of level $k=1,2,...$. We want to solve a Dirichlet problem
in $Q_S$ with $u_k (x,0)\chi_{B(0,2r)}$ on the initial boundary and zero
on the lateral boundary. However, in order to guarantee existence of a (unique and
continuous) solution, we solve a regularized problem
instead.
To this end, we rely on the lower
semicontinuity of $u_k^m (x,0)\chi_{B(0,2r)}$ to approximate it
pointwise from below by Lipschitz functions $\psi_{k,i}^m$, such that $0
\leq \psi_{k,i} \leq \psi_{k,i+1}\leq u_k (x,0)\chi_{B(0,2r)}$ in
$\Omega \times \{0\}$ with $\psi_{k,i}(x) \to u_k (x,0)\chi_{B(0,2r)}$
pointwise in $\Omega$ as $i \to \infty$. That is, we consider the problem
$$
\begin{cases}
\partial_t h_{k,i} - \Delta ( h_{k,i}^m ) = 0 \quad &\text{in } Q_S, \\
h_{k,i}(x,t) = 0 \quad &\text{on } \partial B_{64r} \times (0,S),\\
h_{k,i}(x,0) = \psi_{k,i}(x) \quad &\text{on } \overline{B}(0,64r) \times \{0\}.
\end{cases}
$$
By Theorem~\ref{t.existence_continuous_weaksol} a unique global weak solution $h_{k,i} \in
C(\overline{Q}_S)$ exists such that $h_{k,i} = 0$ on the lateral
boundary and $h_{k,i} = \psi_{k,i}$ on the initial boundary. 
Since $0\leq
\psi_{k,i} \leq u_k \leq u$ on the parabolic boundary, from the
comparison principle in the definition of supercaloric functions it
follows that $0 \leq h_{k,i} \leq u_k \leq u$ in $Q_S$. From
Theorem~\ref{t.existence_continuous_weaksol} it also follows that $0 \leq h_{k,i} \leq
h_{k,i+1}$ in $Q_S$ for every $i$, so that $h_{k,i}$ forms a
nondecreasing sequence with respect to $i\in\N$.
We set
$$
\tilde \theta = \left( \bint_{B_{8r}} h_{k,i}(x,0)\, \d x \right)^{1-m} = 4^{-n(1-m)} \left( \bint_{B_{2r}} \psi_{k,i}(x) \, \d x \right)^{1-m}
$$
and 
$$
\tilde \delta = \left( \frac{|B_1|}{\gamma} \right)^{1-m} \tilde \theta r^2,
$$
where $\gamma$ is the constant from Lemma~\ref{l.l1harnack}. By Lemma~\ref{l.l1linfty_harnack} we have 
$$
\sup_{B_{4r}\times (\frac{\tilde \delta}{2},\tilde \delta)} h_{k,i} \leq \gamma_1 \bint_{B_{2r}} \psi_{k,i}(x) \, \d x,
$$
for $\gamma_1=\gamma_1(n,m)>0$. From Lemma~\ref{l.l1harnack} it follows that 
$$
\inf_{0<\tau<\tilde \delta} \int_{B_{4r}} h_{k,i}(x,\tau) \, \d x \geq \frac{1}{2 \gamma} \int_{B_{2r}} h_{k,i}(x,0) \, \d x.
$$
By using the previous two estimates, we obtain
\begin{align*}
&\frac{1}{2^{n+1}\gamma} \bint_{B_{2r}} h_{k,i}(x,0) \, \d x \\
&\quad\leq \bint_{B_{4r}} h_{k,i}(x,\tau)  \, \d x \\
&\quad= \frac{1}{|B_{4r}|} \int_{\{ h_{k,i}(\cdot,\tau) > c_o \} \cap B_{4r}} h_{k,i}(x,\tau) \, \d x  +  \frac{1}{|B_{4r}|} \int_{\{ h_{k,i}(\cdot,\tau) \leq c_o \} \cap B_{4r}} h_{k,i}(x,\tau)	 \, \d x  \\
&\quad\leq  \frac{ | \{ h_{k,i}(\cdot,\tau) > c_o \} \cap B_{4r} | }{| B_{4r} |} \gamma_1 \bint_{B_{2r}} \psi_{k,i}(x)  \, \d x  +  c_o,
\end{align*}
for any $\tau \in (\frac{\tilde \delta}{2},\tilde \delta)$ and an
arbitrary constant $c_o>0$. By choosing
$$
c_o = \frac{1}{2^{n+2}\gamma} \bint_{B_{2r}} \psi_{k,i}(x) \, \d x,
$$
the estimate above gives
$$
 | \{ h_{k,i}(\cdot,\tau) > c_o \} \cap B_{4r} | \geq \frac{1}{2^{n+2}\gamma \gamma_1} |B_{4r}|
$$
for any $\tau \in (\frac{\tilde \delta}{2},\tilde \delta)$.
At this point we can apply the expansion of positivity, Theorem~\ref{t.expansionofpositivity}. This gives that there exist constants $\eps,\sigma,\eta \in (0,1)$ depending only on $n$ and $m$ such that 
$$
h_{k,i}(\cdot,t) \geq \frac{\eta}{2^{n+2}\gamma} \bint_{B_{2r}} \psi_{k,i}(x) \, \d x \quad \text{ in } B_{8r}
$$
for all $t \in \left[\tau + (1-\eps)\sigma c_o^{1-m}r^2, \tau + \sigma
  c_o^{1-m} r^2 \right]$. Observe that this holds for any $\tau \in
(\frac{\tilde \delta}{2},\tilde \delta)$. Now if we choose the constant $c > 0$ such that
$$
c:= \left( \frac{|B_1|}{4^n \gamma} \right)^{1-m} + \sigma \left( \frac{1}{2^{n+2}\gamma} \right)^{1-m}
$$
and
$$
\theta_{k,i} := c \left( \bint_{B_{2r}} \psi_{k,i}(x)  \, \d x \right)^{1-m},\quad \delta_{k,i} := \theta_{k,i} r^2,
$$
we have that 
\begin{equation}
\inf_{B_{8r} \times (\alpha \delta_{k,i}, \delta_{k,i})} h_{k,i} \geq
\frac{\eta}{2^{n+2}\gamma} \bint_{B_{2r}} \psi_{k,i}(x) \, \d
x,\label{eq:harnack-h}
\end{equation}
where $\alpha\in (0,1)$ depends only on $n$ and $m$.
Moreover, if we first let $i\to\infty$ and then
  $k\to\infty$, by monotone convergence we have
  \begin{equation*}
    \delta_{k,i}\to\delta,
    \qquad\mbox{where\quad}\delta:= c \left( \bint_{B_{2r}} u(x,0)  \, \d x \right)^{1-m}r^2.
  \end{equation*}
  The left-hand side of \eqref{eq:harnack-h} can be estimated from above by using comparison as 
$$
\inf_{B_{8r} \times (\alpha \delta_{k,i}, \delta_{k,i})} h_{k,i} \leq \inf_{B_{8r} \times (\alpha \delta_{k,i}, \delta_{k,i})} u_{k} \leq \inf_{B_{8r} \times (\alpha \delta_{k,i}, \delta_{k,i})} u.
$$
By passing to the limit in \eqref{eq:harnack-h}, first in $i \to \infty$ and then in $k\to \infty$, we obtain
$$
\inf_{B_{8r} \times (\alpha \delta, \delta)} u \geq \frac{\eta}{2^{n+2}\gamma} \bint_{B_{2r}} u(x,0) \, \d x
$$
by using the monotone convergence theorem on the right-hand side. From here the claim follows.

\end{proof}

If $\Omega$ is connected, as a consequence of Lemma~\ref{lem:alternatives} the positivity set
of a supercaloric function in $\Omega_T$ has the form 
$\Omega\times\Lambda_+$, where the set $\Lambda_+\subset(0,T)$ is a
countable union of open time intervals. The next lemma guarantees that the
supercaloric function vanishes at the endpoint $t_o$ of each of these time
intervals, provided $t_o< T$.

\begin{lem} \label{l.endpoint_vanish}
  Let $0< m < 1$, and suppose that $u:\Omega_T\to[0,\infty]$ is a
  supercaloric function in $\Omega_T$ such that for some
  $t_o\in(0,T)$, we have $u(x,t_o) = 0$ for all $x \in \Omega$. Then,
$$
\lim_{t\uparrow t_o} \int_{K} u(x,t) \, \d x = 0
$$
for every $K \Subset \Omega$.
\end{lem}

\begin{proof}
Since an arbitrary compact set $K\Subset\Omega_T$ can be covered by
finitely many balls $B_{r}$ with $B_{4r}\Subset \Omega$, it
suffices to prove the claim for the case $K=B_{r}$ with $B_{4r}\Subset \Omega$.

Let $B(y,4r) \times [s,t_o] \Subset \Omega_T$. Consider the
regularized Dirichlet problem as in the proof of
Lemma~\ref{l.wharnack_supercal} in $\mathcal Q = B(y,4r) \times (s,t_o+\delta)
\Subset \Omega_T$. By using Lemma~\ref{l.l1harnack} together with the
comparison principle $h_{k,i} \leq u_k \leq u$ in $\mathcal Q$ it follows that
\begin{align*}
\int_{B(y,r)} \psi_{k,i}(x) \, \d  x &\leq \gamma \inf_{s<\tau<t_o} \int_{B(y,2r)} h_{k,i}(x,\tau) \, \d x + \gamma \left( \frac{t_o-s}{r^\lambda} \right)^\frac{1}{1-m} \\
&= \gamma \left( \frac{t_o-s}{r^\lambda} \right)^\frac{1}{1-m},
\end{align*}
since $u(\cdot,t_o) \equiv 0$. By using the monotone convergence theorem we can pass to the limit $i \to \infty$ and $k \to \infty$ to obtain
$$
\int_{B(y,r)} u(x,s) \, \d x \leq \gamma \left( \frac{t_o-s}{r^\lambda} \right)^\frac{1}{1-m}.
$$
Since $s < t_o$ was arbitrary, provided that $B(y,4r) \times [s,t_o] \Subset \Omega_T$ holds, we may pass to the limit $s\to t_o$ in the estimate above, from which the claim follows.

\end{proof}

We prove a variant of Lemma~\ref{l.bounded_caccioppoli} when
the supersolution vanishes at the final instant of time. The result will be important in the following section.

\begin{lem} \label{l.bdd_caccioppoli_positivityset}
Let $0<m<1$. Let $u: \Omega_T \to [0,\infty]$ be a supercaloric function in $\Omega_T$ such that $u$ is a weak supersolution in $\Omega \times (t_1,t_2)$ for some interval $(t_1,t_2)\Subset(0,T)$. Furthermore, suppose that $u(x,t_2) = 0$ for every $x \in \Omega$. Then,
$$
\int_{t_1}^{t_2} \int_\Omega \eta^2 |\nabla u^m|^2 \, \d x \d t \leq 4 M^{2m} (t_2-t_1) \int_\Omega |\nabla \eta|^2 \, \d x
$$
for any nonnegative $\eta \in C_0^\infty(\Omega)$ and $M = \|u\|_{L^\infty(\spt(\eta) \times (t_1,t_2))}$.
If $u$ does not vanish at $t_2$, then we have
$$
\int_{t_1}^{t_2} \int_\Omega \eta^2 |\nabla u^m|^2 \, \d x \d t \leq 4 M^{2m} (t_2-t_1) \int_\Omega |\nabla \eta|^2 \, \d x + 2 M^{m+1} \int_\Omega \eta^2 \, \d x. 
$$
\end{lem}

\begin{proof}

We start with a mollified weak formulation for $u$, which can be written as
\begin{align*}
\int_{\tau_1}^{\tau_2} \int_\Omega \partial_t \mollifytime{u}{h}
  \varphi + \mollifytime{\nabla u^m}{h} \cdot \nabla \varphi \, \d x
  \d t &\geq \frac{1}{h}\int_\Omega u(x,\tau_1) \int_{\tau_1}^{\tau_2}
         e^\frac{\tau_1-s}{h} \varphi(x,s) \, \d s \d x
         \ge0
\end{align*}
for a.e. $\tau_2 \in (t_1,t_2)$ and a.e. $\tau_1 \in (t_1,\tau_2)$. The time mollification $\mollifytime{\cdot}{h}$ is defined as in~\eqref{e.time-mollif}. Up next, we use a test function $\varphi = (M^m - u^m) \alpha_{\eps} \eta^2$, where $\eta \in C_0^\infty(\Omega,\R_{\geq 0})$ and $\alpha_{\eps}$ is a piecewise affine approximation of $\chi_{\tau_1,\tau_2}(t)$. For the parabolic part we have 
\begin{align*}
\int_{\tau_1}^{\tau_2} &\int_\Omega \partial_t \mollifytime{u}{h} \varphi \, \d x \d t \\
&= \int_{\tau_1}^{\tau_2} \int_\Omega \alpha_{\eps} \eta^2 M^m \partial_t \mollifytime{u}{h} \, \d x \d t - \int_{\tau_1}^{\tau_2} \int_\Omega \alpha_{\eps} \eta^2 u^m \partial_t \mollifytime{u}{h} \, \d x \d t \\
&\leq \int_{\tau_1}^{\tau_2} \int_\Omega \alpha_{\eps} \eta^2 M^m \partial_t \mollifytime{u}{h} \, \d x \d t - \int_{\tau_1}^{\tau_2} \int_\Omega \alpha_{\eps} \eta^2 \mollifytime{u}{h}^m \partial_t \mollifytime{u}{h} \, \d x \d t \\
&= - \int_{\tau_1}^{\tau_2} \int_\Omega \alpha_{\eps}' \eta^2 M^m \mollifytime{u}{h} \, \d x \d t + \frac{1}{m+1} \int_{\tau_1}^{\tau_2} \int_\Omega \alpha_{\eps}' \eta^2 \mollifytime{u}{h}^{m+1} \, \d x \d t \\
&\xrightarrow{h\to0} - \int_{\tau_1}^{\tau_2} \int_\Omega \alpha_{\eps}' \eta^2 M^m u \, \d x \d t + \frac{1}{m+1} \int_{\tau_1}^{\tau_2} \int_\Omega \alpha_{\eps}' \eta^2 u^{m+1} \, \d x \d t \\
&\xrightarrow{\eps \to 0} - \int_\Omega \eta^2 M^m u(\tau_1) \, \d x + \frac{1}{m+1} \int_\Omega \eta^2 u^{m+1} (\tau_1) \, \d x  \\
&\phantom{\xrightarrow{h\to0}}\; + \int_\Omega \eta^2 M^m u(\tau_2) \, \d x - \frac{1}{m+1} \int_\Omega \eta^2 u^{m+1} (\tau_2) \, \d x
\end{align*}
for a.e. $\tau_2 \in (t_1,t_2)$ and a.e. $\tau_1 \in (t_1,\tau_2)$. Since $\frac{1}{m+1} u^{m+1} \leq M^m u$, the sum of the first two terms on the right-hand side is nonpositive, and we can discard it. After passing to the limit $h \to 0$, for the integrand of the divergence part we have
\begin{align*}
\nabla u^m \cdot \nabla \varphi = - \alpha_{\eps} \eta^2 |\nabla u^m|^2 + 2 \alpha_{\eps} \eta (M^m- u^m) \nabla \eta \cdot \nabla u^m.
\end{align*}
For the latter term we use Young's inequality and obtain
$$
2 \alpha_{\eps} \eta (M^m- u^m) \nabla \eta \cdot \nabla u^m \leq 2\alpha_{\eps} \eta M^m |\nabla \eta | |\nabla u^m| \leq \frac{1}{2} \alpha_{\eps} \eta^2 |\nabla u^m|^2 + 2 \alpha_{\eps}M^{2m} |\nabla \eta|^2.
$$
By passing to the limit $\eps \to 0$ and combining the estimates we have
\begin{align*}
\frac{1}{2}\int_{\tau_1}^{\tau_2} \int_\Omega \eta^2 |\nabla u^m|^2 \, \d x \d t &\leq 2 M^{2m} (\tau_2 - \tau_1)\int_\Omega |\nabla \eta|^2 \, \d x \\
&\phantom{+} + \int_\Omega \eta^2 M^m u(\tau_2) \, \d x - \frac{1}{m+1} \int_\Omega \eta^2 u^{m+1} (\tau_2) \, \d x.
\end{align*}
By multiplying this inequality by 2 and letting $\tau_2 \to t_2$ and
$\tau_1 \to t_1$, the first claim follows by using
Lemma~\ref{l.endpoint_vanish}, while the second one follows by
using $0\le u(\tau_2)\le M$.

\end{proof}

\section{Bounded supercaloric functions} \label{s.connection}

First we state a result concerning the obstacle problem that will have
significant importance in further results of this paper. The existence
and regularity results stated in the following theorem can be extracted from~\cite{BLS,Schaetzler2,Cho_Scheven,MSb}
(see also~\cite{MS}).
The proof of properties (i) and (iv) can be found in \cite{MSo}.
\begin{theo} \label{t.obstacle-supercal}
Let $0<m<1$ and $\Omega \subset \R^n$ be a bounded Lipschitz domain. Let $\psi$ satisfy $\psi^m \in C^{1}(\overline{\Omega_T})$. Then, there exists a function $u \in C(\overline{\Omega_T})$ with the following properties:
\begin{enumerate}
\item[(i)] $u$ is a weak supersolution in $\Omega_T$,
\item[(ii)] $u\geq \psi$ everywhere in $\Omega_T$,
\item[(iii)] $u = \psi$ on $\partial_p \Omega_T$,
\item[(iv)] $u$ is a weak solution in the set $\{ u > \psi \}$.
\end{enumerate}
\end{theo}

We start by proving that supercaloric functions are weak
supersolutions on their positivity set. 

\begin{lem} \label{l.bdd-positive-supercal-is-supersol}
Let $0 < m <1$. Let $u > 0$ be a locally bounded supercaloric function
in $\Omega_T$, where $\Omega \subset \R^n$ is an open set. Then $u$ is a weak supersolution in $\Omega_T$.
\end{lem}

\begin{proof} \
Consider a compactly contained cylinder
$\mathcal{Q} = Q_{t_1,t_2} := B(x_o,r) \times (t_1,t_2) \Subset \Omega_T$
and choose a larger cylinder $\widetilde{\mathcal{Q}}$ with
$\mathcal{Q} \Subset \widetilde{\mathcal{Q}} \Subset \Omega_T$. Observe that by lower
semicontinuity of $u$ and $u>0$ in $\Omega_T$ we have that $u \geq
\delta > 0$ in $\widetilde{\mathcal{Q}}$, for some $\delta>0$. Furthermore, there
exists a sequence $(\psi_k)$ with the properties $\psi_k \in C^\infty(\Omega_T)$ for each $k=1,2,...$, 
$$
0 < \psi_1 < \psi_2 < \cdots < u\ \text{ and }\ \lim_{k\to \infty} \psi_k = u \ \text{ in } \widetilde{\mathcal{Q}}.
$$

Next we consider the obstacle problem in Theorem~\ref{t.obstacle-supercal}, with obstacle $\psi_k$. By Theorem~\ref{t.obstacle-supercal} there exists a solution $v_k \in C(\overline{Q_{t_1,t_2}})$ to the obstacle problem, with $v_k = \psi_k$ on $\partial_p Q_{t_1,t_2}$. In the set 
$$
U_k := \{ (x,t)\in Q_{t_1,t_2} : v_k(x,t) > \psi_k (x,t)   \},
$$
$v_k$ is a weak solution. Since $v_k = \psi_k$ on $\partial_p Q_{t_1,t_2}$, it follows that $v_k = \psi_k$ on $\partial U_k$, except possibly when $t=t_2$. That is,
$$
v_k = \psi_k < u\ \text{ on } \ \partial U_k \cap \{t < t_2 \}.
$$
We want to use now Lemma~\ref{l.comparison_t<T} to conclude that  
\begin{align} \label{e.vlequ}
v_k \leq u\, \text{ in } \ U_k \cap \{ t < t_2 \}.	
\end{align}
Since $v_k$ is continuous in $\overline{Q_{t_1,t_2}}$, it follows that $v_k$ is continuous in $\overline{U_k} \cap \{ t<t_2 \}$. From here it follows that 
$$
\limsup_{U_k \ni (y,s)\to (x,t)} v_k(y,s) = \psi_k(x,t) < u (x,t) \leq \liminf_{U_k \ni (y,s) \to (x,t)} u(y,s)
$$
for each $(x,t) \in \{  (x,t) \in \partial U_k : t < t_2  \}$ by using also lower semicontinuity of $u$. Now we can use Lemma~\ref{l.comparison_t<T} to conclude~\eqref{e.vlequ}.

Consequently, we have that 
$$
\psi_k \leq v_k \leq u \ \text{ in } \ Q_{t_1,t_2},
$$
which implies that $v_k \to u$ as $k\to \infty$ pointwise in
$Q_{t_1,t_2}$. By Lemma~\ref{l.bounded_caccioppoli}, $|\nabla v_k^m|$ is
uniformly bounded in
$L^2(V\times(t_1,t_2))$ for every subdomain $V\Subset B(x_o,r)$. This together with pointwise convergence implies that $\nabla v_k^m$ converges weakly to $\nabla u^m$ in $L^2(V\times(t_1,t_2),\R^n)$. This implies that $u$ is a weak supersolution in any $Q_{t_1,t_2} \Subset \Omega_T$. Since being a weak supersolution is a local property, it follows that $u$ is a weak supersolution in $\Omega_T$. That is, 
$$
\iint_{\Omega_T} \left( - u \partial_t \ph + \nabla u^m \cdot \nabla \ph  \right)  \, \d x \d t \geq 0
$$
for any nonnegative $\ph \in C^\infty_0(\Omega_T)$.

\end{proof}

Next, we generalize the preceding result to nonnegative supercaloric
functions.

\begin{theo} \label{t.bdd_super}
Suppose $0 < m <1$. Let $u \geq 0$ be a locally bounded supercaloric function in $\Omega_T$. Then $u$ is a weak supersolution in $\Omega_T$.
\end{theo}

\begin{proof}

Write $\Omega$ as a union of its connected components, i.e., $\Omega = \bigcup_{j \in \N} \Omega^j$, in which each $\Omega^j$ is open and connected. By Corollary~\ref{c.positivity-intervals} we may decompose the positivity set
$$
  \Lambda_+^j:=\big\{t\in(0,T)\colon u\mbox{\ is positive on }\Omega^j\times\{t\}\big\}
$$
into at most countably many disjoint open intervals $\Lambda_+^j =
\bigcup_{i \in I_j} \Lambda_i^j$, where
$\Lambda_i^j=(t_{i,1}^j,t_{i,2}^j)$.

  On each of the sets $\Omega^j\times\Lambda_i^j$, Lemma~\ref{l.bdd-positive-supercal-is-supersol} implies that $u$ is
  a weak supersolution to~\eqref{evo_eqn}, i.e., $u^m \in L^2_{\loc}(\Lambda_i^j; H^{1}_{\loc}(\Omega^j))$ and
  \begin{equation} \label{e.supersolution-Lambda-i-orig}
    \int_{\Lambda_i^j}\int_{\Omega^j} (-u\partial_t\varphi+\nabla
    u^m\cdot\nabla \varphi)\, \d x\d t\geq0
  \end{equation}
  for all non-negative test functions $\varphi\in
  C^\infty_0(\Omega^j\times\Lambda_i^j)$.

  First we show that $u^m \in L^2_{\loc}(0,T; H^1_{\loc}(\Omega))$. To
  this end, let $K \subset \Omega$ be compact and $(s_1,s_2) \Subset
  (0,T)$. Choose an open set $K'$ such that $K \subset K' \Subset \Omega$ and a cutoff function $\eta \in C^\infty_0(K')$ such that $\eta \equiv 1$ in $K$ and $|\nabla \eta| \leq c \text{ dist} (K,\partial K')^{-1}$ with a numerical constant $c > 0$. Denote $K^j := K \cap \Omega^j$, which is compact since $K_j=K\setminus(\bigcup_{i\neq j}\Omega^i)$ is closed.
  
   For each $\Lambda_i^j \Subset (0,T)$, Lemma~\ref{l.bdd_caccioppoli_positivityset} implies that $u^m \in L^2(\Lambda_i^j;H^1_{\loc}(\Omega^j))$. Denote
$$
I'_j := \{i\in I_j : \Lambda_i^j \cap (s_1,s_2)\neq\varnothing\}.
$$
Observe that for every $t \in (0,T)\setminus \Lambda_+^j$ we have
$u(\cdot,t)\equiv0 $ and $\nabla u^m(\cdot,t) \equiv 0$ on
$\Omega^j$.
By applying Lemma~\ref{l.bdd_caccioppoli_positivityset} on the sets 
$\Omega^j\times(\Lambda_i^j \cap (s_1,s_2))$, we obtain
  \begin{align*}
  \int_{s_1}^{s_2} &\int_{K^j} |\nabla u^m|^2 \, \d x \d t \\
  &\leq
  \int_{s_1}^{s_2} \int_{\Omega^j} \eta^2 |\nabla u^m|^2 \, \d x \d t \\
  &=
  \sum_{i\in I'_j} \int_{\Lambda_i^j \cap (s_1,s_2)} \int_{\Omega^j} \eta^2 |\nabla u^m|^2 \, \d x \d t \\
   &\leq 
   4 M^{2m}  \int_{\Omega^j} |\nabla \eta|^2 \, \d x
   \sum_{i \in I'_j} (t_{i,2}^j-t_{i,1}^j) + 2 M^{m+1} \int_{\Omega^j}
     \eta^2 \, \d x\\
   &\leq 4T M^{2m}
     \int_{\Omega^j} |\nabla \eta|^2 \, \d x
     + 2 M^{m+1} \int_{\Omega^j}  \eta^2 \, \d x
     <\infty
  \end{align*}
  for $M = \|u\|_{L^\infty(K' \times (s_1,s_2))}$, 
  where the last integral can be omitted in the case $s_2\not\in\Lambda_+^j$.
Since $\Omega^j$ and $K^j$ are disjoint and $\Omega = \bigcup_{j\in \N} \Omega^j$, $K = \bigcup_{j\in \N} K^j$, we can sum over $j \in \N$ and obtain
\begin{align*}
  \int_{s_1}^{s_2} &\int_{K} |\nabla u^m|^2 \, \d x \d t \leq 4T M^{2m} \int_{\Omega} |\nabla \eta|^2 \, \d x + 2 M^{m+1} \int_{\Omega}  \eta^2 \, \d x < \infty.
\end{align*}
Since $K$, $s_1$ and $s_2$ were arbitrary, this finally implies that $u^m \in L^2_{\loc}(0,T;H^1_{\loc}(\Omega))$. 

Then we show that the integral inequality~\eqref{e.supersolution-Lambda-i-orig} holds in $\Omega_T$ for all test functions $\varphi \in C^\infty_0(\Omega_T,\R_{\geq0})$. Observe that this implies $\varphi \in C^\infty_0(\Omega^j \times (0,T), \R_{\geq0})$ for every $j\in \N$. Fix $i \in I_j$. For such a test function a standard cut-off argument yields 
  \begin{align*}\label{supersolution-Lambda-i}
    \int_{\tau_1}^{\tau_2}\int_{\Omega^j} (-u\partial_t\varphi+\nabla
    u^m\cdot\nabla \varphi)\, \d x\d t 
    &\geq
        -
    \int_{\Omega^j\times\{\tau_2\}}u\varphi\, \d x
  \end{align*}
  for every $\tau_1 \in \Lambda_i^j$ and a.e. $\tau_2 \in \Lambda_i^j$ with $\tau_2 >
  \tau_1$.
  In the case $t_{i,2}^j<T$, the last term vanishes in the limit
   $\tau_2\uparrow t_{i,2}^j$ due to Lemma~\ref{l.endpoint_vanish}.
   If $t_{i,2}=T$, we only consider test functions that
   vanish in a neighborhood of $\Omega^j\times\{T\}$, so that we can
   omit the last integral also in this case. Since $\varphi$ vanishes also in a neighborhood of $\Omega^j\times\{0\}$, we may pass to the limit $\tau_1 \to t_{i,1}^j$ as well. Thus we get
   \begin{equation*}\label{supersolution-Lambda-i-2}
 \int_{t_{i,1}^j}^{t_{i,2}^j}\int_{\Omega^j} (-u\partial_t\varphi+\nabla
    u^m\cdot\nabla \varphi)\,\d x\d t
    \ge  0.
\end{equation*}
By recalling that $u(\cdot,t)\equiv0 $ and $\nabla u^m(\cdot,t) \equiv 0$ for every $t \in (0,T)\setminus \Lambda_+^j$, we obtain
  \begin{align*}
    \iint_{\Omega^j \times (0,T)}(-u\partial_t\varphi+\nabla u^m\cdot\nabla
    \varphi)\,\d x\d t
    &=
    \sum_{i\in I_j}\int_{\Lambda_i^j}\int_{\Omega^j} (-u\partial_t\varphi+\nabla u^m\cdot\nabla
    \varphi)\,\d x\d t \ge 0.
  \end{align*}
  By summing up over $j \in \N$ and using the fact that $\Omega^j$ are
  disjoint, we conclude the proof.

\end{proof}

We show that a supercaloric function is a weak supersolution also if it belongs to the appropriate energy space.

\begin{lem} \label{l.supercal-in-energyspace}
Let $0<m<1$. Let $u: \Omega_T \to [0,\infty]$ be a supercaloric function in $\Omega_T$ such that $u^m \in L^2_{\loc}(0,T; H^1_{\loc}(\Omega)) \cap L^\frac{1}{m}_{\loc} (\Omega_T)$. Then $u$ is a weak supersolution.
\end{lem}

\begin{proof}
By Theorem~\ref{t.bdd_super}, the truncation $u_k = \min \{u,k\}$ is a weak supersolution for every $k = 1,2,...$, $u_k(x,t) \leq u_{k+1}(x,t)$ and $\lim_{k\to \infty} u_k (x,t) = u(x,t)$ for every $(x,t) \in \Omega_T$. This implies that 
$$
\lim_{k\to \infty} - \iint_{\Omega_T} \partial_t \varphi\, u_k \, \d x \d t = - \iint_{\Omega_T} \partial_t \varphi \,u \, \d x \d t
$$
for every $\varphi \in C^{\infty}_0(\Omega_T,\R_{\geq0})$ by the dominated convergence theorem and the fact that $u \in L^1_{\loc}(\Omega_T)$.

There also holds $\lim_{k\to \infty} \nabla u_k^m(x,t) = \nabla u^m(x,t)$ for a.e. $(x,t) \in \Omega_T$, $|\nabla u_k^m| \leq |\nabla u^m|$ for every $k = 1,2,...$ and $|\nabla u^m | \in L^2_{\loc}(\Omega_T)$. Again, by dominated convergence theorem we can conclude that 
$$
\lim_{k\to\infty} \iint_{\Omega_T} \nabla u_k^m \cdot \nabla \varphi \, \d x \d t = \iint_{\Omega_T} \nabla u^m \cdot \nabla \varphi \, \d x \d t
$$
for every $\varphi \in C_0^\infty(\Omega_T,\R_{\geq0})$, which concludes the proof.
\end{proof}

\section{Barenblatt solutions} \label{sec:barenblatt}

In the case $\frac{n-2}{n} < m < 1$, the Barenblatt solution can be written as
\begin{equation*} 
\mathcal{B}(x,t) = \left( Ct \right)^\frac{1}{1-m} \left(  A t^\frac{2}{\lambda} + |x|^2  \right)^{- \frac{1}{1-m}} \quad \text{ for } (x,t)\in \R^n \times (0,\infty),
\end{equation*}
in which $\lambda = n(m-1)  +2$, $C = 2m\lambda/ (1-m)$ and $A >
0$. The Barenblatt solution is a continuous weak solution in $\R^n \times
(0,\infty)$. However, we may define a function $u$ in the whole space as 
\begin{equation} \label{e.barenblatt2}
u(x,t) =
\begin{cases} 
      \mathcal{B}(x,t), & t > 0, \\
      0, & t\leq 0,
   \end{cases}
\end{equation}
which is not even a weak supersolution in $\R^n\times \R$. That is because the integrability assumption for the gradient fails in any neighbourhood of the origin, i.e. $| \nabla u^m | \notin L^2_{\loc}(\R^n\times \R)$.
However, $u$ is a supercaloric function in the whole space $\R^n
\times \R$. This is due to Lemma~\ref{l.zero-past-extension}, since
$\mathcal{B}$ is a supercaloric function as a continuous weak solution in the upper half-space by Lemma~\ref{l.weasuper-is-supercal}.

The Barenblatt solution is the leading example of a supercaloric function in Barenblatt class that on the other hand is not a weak supersolution.

The Barenblatt solution defined in~\eqref{e.barenblatt2} satisfies 
$$
\partial_t u  - \Delta u^m = M \delta\quad \text{ in } \R^n \times \R
$$
in the weak sense, where $\delta$ is Dirac's delta at the origin and
$M> 0$ represents the mass at the origin ($A$ is a decreasing function of $M$). Furthermore,
$$
\int_{t_1}^{t_2}\int_{B(0,r)} u^{m+ \frac{2}{n}} \, \d x \d t = \infty,
$$
and
$$
\int_{t_1}^{t_2}\int_{B(0,r)} |\nabla u^m|^{1+ \frac{1}{1+mn}} \, \d x \d t = \infty,
$$
for every $t_1 \leq 0$, $t_2 > 0$ and $r > 0$.
Later on, this will show that the integrability exponents obtained in Lemmas~\ref{l.integrability_fun} and~\ref{l.integrability_grad} are sharp.

We interpret
\begin{equation} \label{e.veryweak-gradient}
\nabla u^m  = \lim_{k\to \infty} \nabla \min\{u,k\}^m
\end{equation}
for a supercaloric function $u$. The weak gradient of the truncation is well defined for each $k \in \N$, since $\min \{u,k\}^m \in L^2_{\loc}(0,T;H^1_{\loc}(\Omega))$ by Theorem~\ref{t.bdd_super}. If the gradient defined in~\eqref{e.veryweak-gradient} is a locally integrable function (together with $u^m$), then it is the weak gradient of $u^m$ in the standard sense. Observe that $\nabla u^m = 0$ a.e. in $\{u = \infty\}$, since $\nabla \min\{u,k\}^m = 0$ a.e. in $\{u = \infty\}$ for every $k\in \N$.

We will make use of the following Caccioppoli inequality. For the case $m>1$ see also~\cite[Lemma 2.4]{Pekka_harnack}.

\begin{lem} \label{l.caccioppoli}
Let $0<m<1$. Suppose that $u \geq 0$ is a supercaloric function in $\Omega_T$ and let $\ph \in C_0^\infty(\Omega_T, \R_{\geq 0})$. Then there exist numerical constants $c_1,c_2 > 0$ such that
\begin{align*}
\iint_{\Omega_T}& u^{-m-\eps} |\nabla u^m|^2 \ph^2  \, \d x \d t + \esssup_{t\in (0,T)} \int_\Omega u^{1-\eps} \ph^2 \, \d x \\
&\leq \frac{c_1}{\eps^2} \iint_{\Omega_T} u^{m-\eps} |\nabla \ph|^2 \, \d x \d t + \frac{c_2}{\eps(1-\eps)} \iint_{\Omega_T} u^{1-\eps} |\partial_t (\ph^2)| \, \d x \d t
\end{align*}
holds for every $\eps \in (0,m)$.

\end{lem}

\begin{rem} \label{rem:caccioppoli-grad-interpretatation}
  In points with $u=0$, we interpret the first integrand on the
  left-hand side as zero. This is reasonable since formally for $\eps\in(0,m)$, we have
  \begin{equation*}
    u^{-m-\eps} |\nabla u^m|^2
    =
    \tfrac{4m^2}{(m-\eps)^2}\big|\nabla u^{\frac{m-\eps}{2}}\big|^2,
  \end{equation*}
  and $\frac{m-\eps}{2}>0$. 
\end{rem}

\begin{rem}
The result in Lemma~\ref{l.caccioppoli} holds also if $u$ is a weak supersolution by Theorem~\ref{t.super_lsc} and Lemma~\ref{l.weasuper-is-supercal}.
\end{rem}

\begin{proof}

We again notice that $\Omega = \bigcup_{j\in \N} \Omega^j$, where each $\Omega^j$ is open and connected. First, we consider an arbitrary connected component $\Omega^j$, but denote it by $\Omega$ for simplicity. 
By Corollary~\ref{c.positivity-intervals} we may decompose the positivity set
$$
  \Lambda_+:=\big\{t\in(0,T)\colon u\mbox{\ is positive on }\Omega\times\{t\}\big\}
$$
into at most countably many disjoint open intervals $\Lambda_+ = \bigcup_{i \in I} \Lambda_i$. 

Let $\tau_1, \tau_2 \in \Lambda_i=:(t_{i,1},t_{i,2})$ for some $i\in I$. We consider truncations $u_k = \min \{u,k\}$, $k = 1,2,...$, which
are supercaloric functions with the same positivity set as $u$. For simplicity we denote $u_k$ by $u$. By Lemma~\ref{l.bdd-positive-supercal-is-supersol}, $u$ satisfies the mollified weak formulation
$$
\int_{\tau_1}^{\tau_2} \int_\Omega \varphi \partial_t [u]_h + [\nabla u^m]_h \cdot \nabla \varphi  \, \d x \d t \geq 0
$$
for any nonnegative $\varphi \in
C_0^\infty(\Omega\times(\tau_1,\tau_2))$.
By a standard approximation argument, the same holds more generally for test
functions $\varphi\in C^\infty_0(\Omega_T)$. 
Here
$[\cdot]_h$ denotes the standard mollification in time,
and we consider $h < \frac12 \text{dist}(\partial \Lambda_i,
(\tau_1,\tau_2))$. Observe that in $(\Omega \times
(\tau_1-h,\tau_2+h) ) \cap \spt(\varphi)$
we have $0< \delta \leq u
\leq k < \infty$ for some $\delta>0$.
We test the mollified
formulation with $[u]_h^{-\eps} \varphi^2\in L^2(\tau_1,\tau_2;H^1_0(\Omega))$. From the parabolic part we obtain
\begin{align*}
 \int_{\tau_1}^{\tau_2}& \int_\Omega \varphi^2 [u]_h^{-\eps} \partial_t [u]_h \, \d x \d t \\
&= - \frac{1}{1-\eps} \int_{\tau_1}^{\tau_2} \int_\Omega [u]_h^{1-\eps}\partial_t( \varphi^2) \, \d x \d t +  \frac{1}{1-\eps} \int_\Omega ([u]_h^{1-\eps} \varphi^2) (\cdot,\tau_2) \, \d x \\
&\qquad -  \frac{1}{1-\eps} \int_\Omega ([u]_h^{1-\eps} \varphi^2) (\cdot,\tau_1) \, \d x \\
&\to- \frac{1}{1-\eps} \int_{\tau_1}^{\tau_2} \int_\Omega u^{1-\eps} \partial_t (\varphi^2) \, \d x \d t +  \frac{1}{1-\eps} \int_\Omega (u^{1-\eps} \varphi^2) (\cdot,\tau_2) \, \d x \\
&\qquad -  \frac{1}{1-\eps} \int_\Omega (u^{1-\eps} \varphi^2) (\cdot,\tau_1) \, \d x
\end{align*}
as $h \to 0$, for a.e. $\tau_1<\tau_2$ in $\Lambda_i$. Observe also
that the second term on the right hand side converges to $0$ when
$\tau_2 \to t_{i,2}$.

For the gradient we have
$$
\nabla ([u]_h^{-\eps} \varphi^2) = 2 \varphi [u]_h^{-\eps} \nabla \varphi - \frac{\eps}{m} \varphi^2 [u]_h^{-\eps - 1} [u^{1-m} \nabla u^m]_h.
$$
Observe that since $0<\delta \leq u \leq k$, we also have $\delta \leq
[u]_h \leq k$. Now each mollified term above converges
pointwise a.e.~when $h \to 0$. In particular, the last term is
majorized by
$$
\varphi^2 [u]_h^{-\eps - 1} |[u^{1-m} \nabla u^m]_h| |[\nabla u^m]_h| \leq \delta^{-\eps-1} k^{1-m} \chi_{\spt(\varphi)} \| \varphi \|_\infty^2 \big[|\nabla u^m|\big]_h^2,
$$
and for the integral of the majorant, we have the convergence
\begin{align*}
\lim_{h\to 0} &\int_{\tau_1}^{\tau_2} \int_\Omega \delta^{-\eps-1} k^{1-m} \chi_{\spt(\varphi)}\| \varphi \|_\infty^2 \big[|\nabla u^m|\big]_h^2 \, \d x \d t \\
&= \int_{\tau_1}^{\tau_2} \int_\Omega \delta^{-\eps-1} k^{1-m} \chi_{\spt(\varphi)}\| \varphi \|_\infty^2 |\nabla u^m|^2 \, \d x \d t<\infty,
\end{align*}
since $[\nabla u^m]_h \to \nabla u^m$ in $L^2_{\loc}(\Omega_T)$ when $h \to 0$.
Thus, we can use a variant of the dominated convergence theorem~\cite[Theorem 4, Chapter 1.3]{EG} to conclude
$$
\lim_{h\to 0} \int_{\tau_1}^{\tau_2} \int_\Omega \varphi^2 [u]_h^{-\eps - 1} [\nabla u^m]_h \cdot [u^{1-m} \nabla u^m]_h  \, \d x \d t =  \int_{\tau_1}^{\tau_2} \int_\Omega \varphi^2 u^{-m-\eps} |\nabla u^m|^2 \, \d x \d t.
$$
We can argue similarly with the other term in the divergence part, which implies 
\begin{align*}
\int_{\tau_1}^{\tau_2} \int_\Omega [\nabla u^m]_h \cdot \nabla([u]_h^{-\eps} \varphi^2)  \, \d x \d t &\to 2 \int_{\tau_1}^{\tau_2} \int_\Omega \varphi u^{-\eps} \nabla u^m \cdot \nabla \varphi \, \d x \d t \\
&\phantom{+}- \frac{\eps}{m} \int_{\tau_1}^{\tau_2} \int_\Omega \varphi^2 u^{-m-\eps} |\nabla u^m|^2 \, \d x \d t
\end{align*}
when $h \to 0$. By Young's inequality we have 
\begin{align*}
2 \varphi u^{-\eps}\nabla u^m \cdot \nabla \varphi \leq \frac{\eps}{2m} \varphi^2 u^{-m-\eps} |\nabla u^m|^2 + \frac{2m}{\eps} u^{m- \eps} |\nabla \varphi|^2.
\end{align*}
By combining the results we obtain
\begin{align*}
\frac{\eps}{2m} &\int_{\tau_1}^{\tau_2} \int_\Omega \varphi^2 u^{-m-\eps} |\nabla u^m|^2  \, \d x \d t + \frac{1}{1-\eps} \int_\Omega (u^{1-\eps} \varphi^2) (\cdot,\tau_1) \, \d x \\
&\leq \frac{2m}{\eps} \int_{\tau_1}^{\tau_2} \int_\Omega u^{m-\eps} |\nabla \varphi|^2 \, \d x \d t + \frac{1}{1-\eps} \int_{\tau_1}^{\tau_2} \int_\Omega u^{1-\eps} |\partial_t (\varphi^2)| \, \d x \d t \\
&\phantom{+} +  \frac{1}{1-\eps} \int_\Omega (u^{1-\eps} \varphi^2) (\cdot,\tau_2) \, \d x.
\end{align*}
Now we can pass to the limit $\tau_2 \to t_{i,2}$ so that the last
term vanishes due to Lemma~\ref{l.endpoint_vanish} if $t_{i,2} < T$ and also in the case $t= T$ since $\varphi$ vanishes in a neighborhood of $\Omega \times \{T\}$. On the right hand side we may integrate
over $\Omega \times \Lambda_i$. At this point,
we also pass to the limit $k \to \infty$ in the truncations. Using
Fatou's lemma for the first term on the left-hand side and the monotone convergence
theorem for the remaining terms, we obtain the inequality above for the original
function $u$. Observe that if the right hand side tends to infinity, the estimate clearly holds. Thus we may assume that the right hand side is finite. By considering separately the terms on the left-hand
side, in the first term we can pass to the limit $\tau_1 \to
t_{i,1}$. In the second term on the left-hand side, we take the
supremum over $\tau_1 \in \Lambda_i$. In this way, we arrive at the bound
\begin{align*}
\iint_{\Omega \times \Lambda_i} &\varphi^2 u^{-m-\eps} |\nabla u^m|^2  \, \d x \d t + \esssup_{t \in \Lambda_i} \int_\Omega (u^{1-\eps} \varphi^2) (\cdot,t) \, \d x \\
&\leq \frac{4m^2+2m\eps(1-\eps)}{\eps^2} \iint_{\Omega \times \Lambda_i} u^{m-\eps} |\nabla \varphi|^2 \, \d x \d t \\
&\phantom{+} + \frac{2m+\eps-\eps^2}{\eps (1-\eps)} \iint_{\Omega \times \Lambda_i} u^{1-\eps} |\partial_t (\varphi^2)| \, \d x \d t.
\end{align*}

Observe that in $\Omega_T \setminus \left( \bigcup_{i \in I} \Omega
  \times \Lambda_i \right)$ both sides are zero since in this set $u
\equiv 0$ and $\nabla u^m \equiv 0$, see also Remark~\ref{rem:caccioppoli-grad-interpretatation}.
By summing up over $i\in I$ we have
\begin{align*}
\iint_{\Omega_T} &u^{-m-\eps} |\nabla u^m|^2 \varphi^2 \, \d x \d t + \esssup_{t \in (0,T)} \int_\Omega u(x,t)^{1-\eps}\varphi(x,t)^2 \, \d x \\
&\leq \frac{c_1}{\eps^2} \iint_{\Omega_T} u^{m-\eps} |\nabla \varphi|^2 \, \d x \d t + \frac{c_2}{\eps(1-\eps)} \iint_{\Omega_T} u^{1-\eps} |\partial_t (\varphi^2)| \, \d x \d t,
\end{align*}
for numerical constants $c_1,c_2 > 0$. In the end, we may sum up over all connected components of $\Omega$, which concludes the proof. 

\end{proof}

We recall Sobolev's inequality, see~\cite{KuLiPa, DGV}.

\begin{lem}[Sobolev]\label{l.sobolev}
	Assume that $w\in L^p_{\loc}(0,T;W_{\loc}^{1,p}(\Omega))$ and
        $\varphi\in C^\infty_0(\Omega_T)$, and $r>0$. There exists a constant $c=c(n,p,r)$ such that the inequality
        \begin{equation}
	\iint_{\Omega_T}|\varphi w|^q \d x\d t\leq
        c^q\iint_{\Omega_T}|\nabla(\varphi w)|^p\d x \d
        t\left(\esssup_{0<t<T}\int_\Omega|\varphi w|^r \d
          x\right)^{\frac{p}{n}},\label{eq:sobolev}
      \end{equation}
is valid for $q=p+\frac{pr}{n}$.  
\end{lem}

Up next we prove a local integrability result for supercaloric functions by exploiting a Moser type iteration.

\begin{lem} \label{l.integrability_fun}
Let $\frac{n-2}{n} < m < 1$ and $\Omega$ be an open set in $\R^n$. Suppose that $u$ is a nonnegative supercaloric function in $\Omega_T$. If $u \in L^s_{\loc}(\Omega_T)$ for some $s > \frac{n}{2}(1-m)$, then $u \in L^q_{\loc}(\Omega_T)$ whenever $q < m + \frac{2}{n}$.
\end{lem}

\begin{proof}

By Theorem~\ref{t.bdd_super}, the truncations $u_k:=\min\{u,k\}$
are weak supersolutions for any  $k>0$ and satisfy the Caccioppoli
estimate in Lemma~\ref{l.caccioppoli}. Up next, we combine
Sobolev inequality, Lemma~\ref{l.sobolev} and Caccioppoli
inequality, Lemma~\ref{l.caccioppoli}. 

Let $\ph \in
C_0^\infty(\Omega_T)$, $0\leq \ph \leq 1$ and $\ph =1 $ in a compact
subset of $\Omega$.
Since $m > \frac{n-2}{n}$, it follows that $\frac{n}{2}(1-m) <
1$. Therefore, there exists $\eps \in (0,m)$ with $s =
1-\eps>\frac{n}{2}(1-m)$. We choose
$$
w = u_{k}^\frac{s-(1-m)}{2}=u_k^{\frac{m-\eps}{2}},\quad p=2\quad\mbox{ and }\quad
r= \frac{2s}{s-(1-m)}>2
$$
in Sobolev inequality, and start to estimate the right hand side.  For the first term we have
\begin{align*}
\iint_{\Omega_T} &\left| \nabla \left( \ph u_{k}^\frac{m-1 +s}{2}  \right)  \right|^2  \, \d x \d t \\
&\leq 2 \iint_{\Omega_T} u_{k}^{m-1 + s} |\nabla \ph|^2 \, \d x \d t +
  c \iint_{\Omega_T}  u_{k}^{-m-1+s} |\nabla(u_k^m)|^2\ph^2 \, \d x \d t \\
&\leq c \iint_{\Omega_T} u_{k}^{m-1 + s} |\nabla \ph|^2 \, \d x \d t + c \iint_{\Omega_T} u_{k}^{s} |\partial_t ( \ph^2)| \, \d x \d t,
\end{align*}
in which $c = c(m,\eps)>0$. In the last step we applied the Caccioppoli inequality from
Lemma~\ref{l.caccioppoli} with $\eps=1-s$.
With the aforementioned lemma we can also estimate the second term from Sobolev
inequality~\eqref{eq:sobolev}. Since 
$r > 2$, the function $\ph^\frac{r}{2} \in C^1_0(\Omega_T)$
is an admissible test function in the Caccioppoli inequality, which gives
\begin{align*}
&\esssup_{t\in (0,T)} \int_\Omega u_{k}^s (\ph^\frac{r}{2})^2 \, \d x \\
&\leq c \iint_{\Omega_T} u_{k}^{m-1 + s} \ph^{r-2} |\nabla \ph|^2 \, \d x \d t + c \iint_{\Omega_T} u_{k}^{s} \ph^{r-2} |\partial_t ( \ph^2)| \, \d x \d t \\
&\leq c \iint_{\Omega_T} u_{k}^{m-1 + s} |\nabla \ph|^2 \, \d x \d t + c \iint_{\Omega_T} u_{k}^{s} |\partial_t ( \ph^2)| \, \d x \d t,
\end{align*} 
where $c= c(m,\eps)>0$.
By using the Sobolev inequality, Lemma~\ref{l.sobolev}, and the two inequalities above, we obtain
\begin{align*}
\iint_{\Omega_T} &\ph^q u_{k}^{s\left( 1+ \frac{2}{n} \right)- (1-m)} \, \d x \d t \\
&\leq \left( c \iint_{\Omega_T} u_{k}^{m-1 + s} |\nabla \ph|^2 \, \d x \d t + c \iint_{\Omega_T} u_{k}^{s} |\partial_t ( \ph^2)| \, \d x \d t  \right)^{1+\frac{2}{n}},
\end{align*}
with a constant $c=c(n,m,\eps)>0$. We can estimate
\begin{align*}
\iint_{\Omega_T} &u_{k}^{m-1 + s} |\nabla \ph|^2 \, \d x \d t \\
&= \iint_{\Omega_T} \chi_{\{u_k>1 \}} u_{k}^{m-1 + s} |\nabla \ph|^2 \, \d x \d t + \iint_{\Omega_T} \chi_{\{u_k\leq1 \}} u_{k}^{m-1 + s} |\nabla \ph|^2 \, \d x \d t \\
&\leq \iint_{\Omega_T}  u_{k}^{s} |\nabla \ph|^2 \, \d x \d t + \iint_{\Omega_T}  |\nabla \ph|^2 \, \d x \d t, 
\end{align*}
and further
\begin{align*}
\iint_{\Omega_T} &\ph^q u_{k}^{s\left( 1+ \frac{2}{n} \right)- (1-m)} \, \d x \d t \\
&\leq c(n,m,\eps)\left( \iint_{\Omega_T} u_{k}^{s} \left( |\nabla \ph|^2 + |\partial_t ( \ph^2)|\right) \, \d x \d t + \iint_{\Omega_T} |\nabla \ph|^2 \, \d x \d t  \right)^{1+\frac{2}{n}} \\
&\leq c(n,m,\eps)\left( \iint_{\Omega_T} u^{s} \left( |\nabla \ph|^2 + |\partial_t ( \ph^2)|\right) \, \d x \d t + \iint_{\Omega_T} |\nabla \ph|^2 \, \d x \d t  \right)^{1+\frac{2}{n}}.
\end{align*}
Now we can pass to the limit $k \to \infty$ and use monotone convergence theorem on the left-hand side, which implies
$$
u \in L^{s\left( 1+ \frac{2}{n} \right)- (1-m)}_{\loc}(\Omega_T).
$$
We can repeat this procedure as long as $\eps > 0$, i.e., the integrability exponent is
strictly less than~$1$.
By iteration we obtain a sequence of integrability exponents
$$
s_i = s_{i-1} \left( 1 + \frac{2}{n} \right)- (1-m),
$$
provided $s_{i-1}<1$. The exponents can be written in terms of the integrability exponent
$s_0=1-\eps>\frac n2(1-m)$ as
$$
s_i =  \left( 1 + \frac{2}{n} \right)^i \left( s_0 - \frac{n}{2}(1-m)
\right) + \frac{n}{2} (1-m).
$$
In a finite number of iteration steps we obtain the integrability $u
\in L^1_{\loc}(\Omega_T)$. Then, we let $\sigma\in(0,m)$ and $s =1-\frac{\sigma}{1+\frac{2}{n}}$. Combining
Sobolev and Caccioppoli inequalities once more we obtain
$$
 u \in L^{m+ \frac{2}{n} - \sigma}_{\loc}(\Omega_T).
$$
Since $\sigma\in(0,m)$ is arbitrary, the claim follows.
\end{proof}

Next we prove a local integrability result for the gradient $\nabla u^m$.

\begin{lem} \label{l.integrability_grad}
Let $\frac{n-2}{n}< m <1$ and let $\Omega \subset \R^n$ be an open set. Suppose that $u$ is a nonnegative supercaloric function with $u \in L^s_{\loc}(\Omega_T)$ for some $s > \frac{n}{2}(1-m)$. Then, the weak gradient $\nabla u^m$ exists and  $\left| \nabla u^m \right| \in L^q_{\loc}(\Omega_T)$ for any $q < 1 + \frac{1}{1+mn}$.
\end{lem}

\begin{proof} \
By Lemma~\ref{l.integrability_fun}, it already follows that $u \in
L^r_{\loc}(\Omega_T)$ whenever $r < m + \frac{2}{n}$. In particular,
$u \in L^1_{\loc}(\Omega_T)$. First we start with truncations
$u_k=\min\{u,k\}$. Let $\Omega' \Subset \Omega$, $0<t_1<t_2<T$ and
$\eps \in (0,m)$. By Theorem~\ref{t.bdd_super}, the truncation $u_{k}$
is a weak supersolution for every $k\in\N$. Now for $q<1+\frac{1}{1+mn}$
and $\varphi \in C^{\infty}_0(\Omega_T)$ with $\varphi = 1$ in $\Omega' \times (t_1,t_2)$ and $\varphi \geq 0$ we have
\begin{align*}
&\int_{t_1}^{t_2} \int_{\Omega'} \left| \nabla u_{k}^m \right|^q  \, \d x \d t \\
&= \int_{\Lambda_+\cap (t_1,t_2)} \int_{\Omega'} \left| \nabla u_{k}^m \right|^q  \, \d x \d t
\\
&= \int_{\Lambda_+\cap (t_1,t_2)} \int_{\Omega'}\left( u_{k}^{- \frac{m+\eps}{2}} \left| \nabla u_{k}^m \right| \right)^q u_{k}^{q \frac{m+\eps}{2}} \, \d x \d t \\
&\leq \left( \int_{t_1}^{t_2} \int_{\Omega'} u_{k}^{-m-\eps} \left| \nabla u_{k}^m \right|^2 \, \d x \d t \right)^\frac{q}{2} \left( \int_{t_1}^{t_2} \int_{\Omega'} u_{k}^{\frac{q}{2-q} (m+\eps )}  \, \d x \d t \right)^{1-\frac{q}{2}} \\
&\leq  \left(c \iint_{\Omega_T} \left( u_{k}^{m-\eps} |\nabla \varphi|^2 + u_{k}^{1-\eps} |\partial_t (\varphi^2)| \right) \, \d x \d t \right)^\frac{q}{2} \left( \int_{t_1}^{t_2} \int_{\Omega'} u_{k}^{\frac{q}{2-q} (m+\eps )}  \, \d x \d t \right)^{1-\frac{q}{2}} \\
&\leq \left(c \iint_{\Omega_T} \left( u^{m-\eps} |\nabla \varphi|^2 + u^{1-\eps} |\partial_t (\varphi^2)| \right) \, \d x \d t \right)^\frac{q}{2} \left( \int_{t_1}^{t_2} \int_{\Omega'} u^{\frac{q}{2-q} (m+\eps )}  \, \d x \d t \right)^{1-\frac{q}{2}} 
\end{align*}
for $c= c(\eps) >0$ by using H\"older's inequality and the Caccioppoli inequality, Lemma~\ref{l.caccioppoli}. The first integral on the right hand side is clearly bounded since
$u\in L^1_{\mathrm{loc}}(\Omega_T)$, and the second is as well
whenever $\frac{q}{2-q} (m+\eps ) < m + \frac{2}{n}$ by
Lemma~\ref{l.integrability_fun}. Since $\eps >0$ can be chosen
arbitrarily small, the second integral is finite whenever $ q < 1 + \frac{1}{1+mn}$, which completes the proof.
\end{proof}

\begin{rem}
Observe that in the case $0<m \leq \frac{n-2}{n}$ (and in particular when $m =\frac{n-2}{n}$) the proof of the preceding lemma also implies that if $u \in L^1_{\loc}(\Omega_T)$ is a supercaloric function in $\Omega_T$, then $|\nabla u^m|\in L^q_{\loc}(\Omega_T)$ for every $q < \frac{2}{m+1}$. Indeed, in that case $\frac{2}{m+1} > 1$ and $\nabla u^m$ is a weak gradient of $u^m$.
\end{rem}

Finally, we state characterizations for Barenblatt type supercaloric functions.

\begin{theo} \label{t.barenblatt}
Let $\frac{n-2}{n} < m < 1$ and $\Omega$ be an open set in $\R^n$. Suppose that $u$ is a nonnegative supercaloric function in $\Omega_T$. Then the following statements are equivalent:

\begin{itemize}
\item[(i)] $u\in L^q_{\loc}(\Omega_T)$ for some $q > \frac{n}{2}(1-m)$,
\item[(ii)] $u \in L^{\frac{n}{2}(1-m)}_{\loc}(\Omega_T)$,
\item[(iii')] there exists $\alpha \in \left(\frac{n}{2}(1-m),1\right)$ such that
$$
\sup_{\delta< t< T-\delta} \int_{\Omega'} u(x,t)^\alpha \, \d x < \infty,
$$
whenever $\Omega' \Subset \Omega$ and $\delta \in (0,\frac{T}{2})$,
\item[(iii)] 
$$
\sup_{\delta< t< T-\delta} \int_{\Omega'} u(x,t) \, \d x < \infty,
$$
whenever $\Omega' \Subset \Omega$ and $\delta \in (0,\frac{T}{2})$.
\end{itemize}
\end{theo}

\begin{proof}
(i) $\implies$ (ii): H\"older inequality. \\
(iii') $\implies$ (i): Elementary. \\
(i) $\implies$ (iii'): This is a direct consequence of the Caccioppoli inequality, Lemma~\ref{l.caccioppoli}.
(iii) $\implies$ (iii'): H\"older inequality. \\
(ii) $\implies$ (i): Follows from proving contraposition $\neg$(i) $\implies$ $\neg$(ii) in Theorem~\ref{t.complementary}. \\
(i) $\implies$ (iii): Follows from proving contraposition $\neg$(iii) $\implies$ $\neg$(i) in Theorem~\ref{t.complementary}. 
\end{proof}

Observe that every supercaloric function $u$ in the Barenblatt class satisfies
\begin{align*}
0 &\leq \lim_{k\to \infty} \iint_{\Omega_T}- u_k \partial_t \varphi  + \nabla u_k^m \cdot \nabla \varphi \, \d x \d t = \iint_{\Omega_T}- u \partial_t \varphi + \nabla u^m \cdot \nabla \varphi \, \d x \d t,
\end{align*}
for every nonnegative $\varphi \in C_0^\infty(\Omega_T)$ by Theorem~\ref{t.bdd_super} and Lemmas~\ref{l.integrability_fun},~\ref{l.integrability_grad}. Together with Riesz' representation theorem this implies that for every supercaloric function $u$ in the Barenblatt class there exists a nonnegative Radon measure $\mu$ in $\Omega_T$ such that 
$$
\iint_{\Omega_T}- u \partial_t \varphi + \nabla u^m \cdot \nabla \varphi \, \d x \d t = \int_{\Omega_T} \varphi\, \d \mu.
$$
for every $\varphi \in C_0^\infty(\Omega_T)$.

\section{Infinite point-source solutions} \label{sec:ipss}

In this section, we consider supercaloric functions that do not fall into the class described by Theorem~\ref{t.barenblatt}. As a starting point, we recall that a function
\begin{equation}\label{eq:elliptic-f}
u(x,t) = |x|^{- \frac{n-2}{m}},\quad \text{ for } n\geq 3,\quad 0<m<1,
\end{equation}
based on the fundamental solution to the elliptic (Laplace) equation is a supercaloric function to the porous medium equation in the whole space $\R^{n+1}$. In the supercritical case the singularity of the function in~\eqref{eq:elliptic-f} is mild enough to guarantee that it belongs to the Barenblatt class. However, $|\nabla u^m| \notin L^2_{\loc}(\R^{n+1})$, which implies that $u$ is not a weak supersolution.

For the rest of this section, we focus only on the supercritical range
$\frac{n-2}{n}<m<1$. In the complementary class, the leading example
is the infinite point-source solution, which possesses a slightly similar behavior as~\eqref{eq:elliptic-f}. The infinite point-source solution (see~\cite{CV}) can be written as 
\begin{equation} \label{e.IPSS}
 \mathcal{U}(x,t) = \left( \frac{C t}{|x|^2} \right)^\frac{1}{1-m}, \qquad C = \frac{2m}{1-m}\left( 2 - n(1-m) \right) > 0,
\end{equation}
for $(x,t)\in \R^n\times (0,\infty)$.
This function is a continuous weak solution to~\eqref{evo_eqn} in
$\left( \R^n\setminus \{ 0 \} \right) \times (0,\infty)$. However, $u \notin L_{\loc}^{\frac{n}{2}(1-m)}(\R^n\times (0,\infty))$ which implies that $u$ is not even an integrable function in $\R^n\times (0,\infty)$. However $\mathcal{U}$ is a supercaloric function in $\R^n \times (0,\infty)$ which we show in the next lemma.

\begin{lem}
The infinite point-source solution $\mathcal{U}$ defined in~\eqref{e.IPSS} is a supercaloric function in $\R^n \times (0,\infty)$.
\end{lem}

\begin{proof}

Denote $\mathcal{U}_k = \min\{\mathcal{U},k\}$. Now $\mathcal{U}_k$ is clearly continuous in $\R^n \times (0,\infty)$ and a supercaloric function in $(\R^n \setminus \{0\}) \times (0,\infty)$ as a truncation of continuous weak solution. Let $Q_{t_1,t_2}=Q\times(t_1,t_2) \Subset \R^n \times (0,\infty)$ be a $C^{2,\alpha}$-cylinder such that $0\in Q$ and $h \in C(\overline{Q_{t_1,t_2}})$ be a weak solution in $Q_{t_1,t_2}$ with $h \leq \mathcal{U}_k$ on $\partial_p Q_{t_1,t_2}$. This immediately implies that $h \leq k$ in $Q_{t_1,t_2}$ and in particular $h \leq \mathcal{U}_k = k$ on $\{0\} \times [t_1,t_2)$. Since $h$ is subcaloric we can use Lemma~\ref{l.supersubcal-cylinder-comparison} to conclude that also $h \leq \mathcal{U}_k$ in $(Q\setminus \{0\}) \times (t_1,t_2)$. 

If $0 \in \partial Q$ we can use the fact that $\mathcal{U}_k = k$ in
$\overline{B_r(0)} \times (t_1,t_2)$ with $r = \left( \frac{C
    t_1}{k^{1-m}} \right)^\frac{1}{2}$. Since $h \leq k$ in
$Q_{t_1,t_2}$, it follows that $h \leq \mathcal{U}_k$ in $(\overline{B_r(0)}
\cap Q) \times (t_1,t_2)$ with the previously defined $r$. In the set
$(Q \setminus \overline{B_r(0)} ) \times (t_1,t_2)$ we can use
Lemma~\ref{l.supersubcal-cylinder-comparison} to conclude that $h
\leq \mathcal{U}_k$ in $(Q\setminus\overline{B_r(0)}) \times (t_1,t_2)$, and
therefore in the whole cylinder $Q_{t_1,t_2}$. Thus $\mathcal{U}_k$ is supercaloric in $\R^n \times (0,\infty)$.

By Lemma~\ref{l.superc_increasing_lim}, also the pointwise limit $\lim_{k\to \infty} \mathcal{U}_k  = \mathcal{U}$ is supercaloric in $\R^n \times (0,\infty)$. 

\end{proof}

Again, zero extension of $\mathcal U$ to nonpositive times $t \leq 0$, say $u$, is supercaloric in $\R^n \times \R$ by Lemma~\ref{l.zero-past-extension}. However, $u \notin L_{\loc}^{\frac{n}{2}(1-m)}(\R^n\times \R)$.

We can modify the example above to obtain supercaloric functions with as bad singularity as we please. We can define
$$
\mathcal U(x,t) = \left( \frac{C t}{|x|^q} \right)^\frac{1}{1-m},
$$
in which $q \geq 2$ can be as large as we wish and 
$$
C = qm\left( 2+\frac{qm}{1-m} - n \right).
$$
This is still a supercaloric function in $B(0,1) \times (0,\infty)$. However, for given $\eps >0$, $\mathcal U \notin L^\eps_{\loc}$ if $q \geq \frac{n}{\eps}(1-m)$.

Before stating characterizations in the complementary class, we state and prove an auxiliary result, which is analogous to~\cite[Lemma 4.5]{GKM_supercal}.

\begin{lem} \label{l.blowup}
Let $\frac{n-2}{n}<m<1$. Let $u$ be a supercaloric function in $\Omega_T$. Suppose that there exists a point $x_o \in \Omega$ and a sequence $(t_j)$ in $(0,T)$ with $t_j \to t_o \in (0,T)$ as $j\to \infty$, such that 
$$
\lim_{j\to\infty} \int_{B(x_o,r)} u(x,t_j) \, \d x = \infty
$$
whenever $ r > 0$ and $B(x_o,r) \Subset \Omega$. Then, 
$$
\liminf_{(x,s)\to (x_o,t)} u(x,s)\left| x-x_o \right|^\frac{2}{1-m} > 0
$$
for every $t > t_o$.
\end{lem}

\begin{proof}\
Fix $r>0$ with $B(x_o,64r)\Subset\Omega$ and let $t\in (t_o,T)$. Then, for large enough $j$ we have that 
$$
\bint_{B(x_o,r)} u(x,t_j) \, \d x \geq 4 c \left( \frac{t-t_j}{r^2} \right)^\frac{1}{1-m}, 
$$
where $c = c(n,m)$ is the constant from Lemma~\ref{l.l1harnack} with integral averages. There exist truncations $u_{k_j}:=\min\{u,k_j\}$ such that
\begin{align} \label{e.trunc_eq}
\bint_{B(x_o,r)} u_{k_j}(x,t_j) \, \d x = 2 c \left( \frac{t-t_j}{r^2} \right)^\frac{1}{1-m}.
\end{align}
By lower semicontinuity of $u_{k_j}$, there exists a sequence of
Lipschitz functions $(\psi_{k_j,i})_{i\in\N}$, such that $0 \leq
\psi_{k_j,i} \leq \psi_{k_j,i+1} \leq u_{k_j}^m$ and $\psi_{k_j,i} \to
u_{k_j}^m$ pointwise in $\Omega_T$ as $i \to \infty$. By
Theorem~\ref{t.existence_continuous_weaksol}, there exists a unique continuous solution
$h_{k_j,i} \in C(\overline{B(x_o,2r) \times (t_j,T)})$, such that
$h_{k_j,i} = \psi_{k_j,i}^\frac{1}{m}$ on the parabolic boundary of
$B(x_o,2r) \times (t_j,T)$.
By the comparison principle from the
definition  of supercaloric functions, it then follows that $h_{k_j,i} \leq u_{k_j}$ for every $i\in\N$. By taking $s = t_j$ and $t<T$ in Lemma~\ref{l.l1harnack}, we have that  
\begin{align*}
\sup_{t_j<\tau<t} \bint_{B(x_o,r)} h_{k_j,i}(x,\tau)\, \d x &\leq c \inf_{t_j<\tau<t} \bint_{B(x_o,2r)} h_{k_j,i}(x,\tau)\, \d x + c\left( \frac{t-t_j}{r^2} \right)^\frac{1}{1-m} \\
&\leq c \inf_{t_j<\tau<t} \bint_{B(x_o,2r)} u_{k_j}(x,\tau)\, \d x + c\left( \frac{t-t_j}{r^2} \right)^\frac{1}{1-m},
\end{align*}
where comparison was used in the second inequality. The left-hand side we can further bound from below as
$$
\sup_{t_j<\tau<t} \bint_{B(x_o,r)} h_{k_j,i}(x,\tau)\, \d x \geq \bint_{B(x_o,r)} h_{k_j,i}(x,t_j)\, \d x = \bint_{B(x_o,r)} \psi_{k_j,i}^\frac{1}{m}(x,t_j)\, \d x.
$$
By combining the inequalities above and passing to the limit $i \to \infty$ and using~\eqref{e.trunc_eq}, we obtain
\begin{align*}
2 c \left( \frac{t-t_j}{r^2} \right)^\frac{1}{1-m} &= \bint_{B(x_o,r)} u_{k_j}(x,t_j) \, \d x = \lim_{i\to \infty} \bint_{B(x_o,r)} \psi_{k_j,i}^\frac{1}{m}(x,t_j)\, \d x \\
&\leq c\,\bint_{B(x_o,2r)} u(x,\tau)\, \d x + c\left( \frac{t-t_j}{r^2} \right)^\frac{1}{1-m}
\end{align*}
for any $\tau \in (t_j,t)$ and large enough $j$. From here it follows that 
$$
\left( \frac{t-t_j}{r^2} \right)^\frac{1}{1-m} \leq \bint_{B(x_o,2r)} u(x,\tau)\, \d x.
$$
By passing to the limit $j \to \infty$, this implies 
$$
r^2 \left( \bint_{B(x_o,2r)} u(x,\tau)\, \d x \right)^{1-m} \geq t-t_o, 
$$
for any $\tau \in (t_o,t)$. Observe that $r>0$ was arbitrary. By taking any sequence $(r_j)$ with $0<r_j \to 0$ as $j \to \infty$, we have 
$$
\liminf_{j\to \infty} r_j^2 \left( \bint_{B(x_o,2r_j)} u(x,\tau)\, \d x \right)^{1-m} \geq t -t_o > 0
$$
for any $\tau \in (t_o,t)$. For the constant $c_2=c_2(n,m)$ from
Lemma~\ref{l.wharnack_supercal}, we fix $\eps \in (0,\min\{c_2 (t-t_o), T-t\} )$, $\tau \in (t_o,t)$ and choose truncation levels $k_j$ such that 
$$
c_2 r_j^2 \left( \bint_{B(x_o,2r_j)} u_{k_j}(x,\tau)\, \d x \right)^{1-m} = \eps
$$
holds for all large enough $j$. Now we can apply Lemma~\ref{l.wharnack_supercal} and obtain 
\begin{align*}
\inf_{B(x_o,2r_j)} u(\cdot,s) \geq \inf_{B(x_o,2r_j)} u_{k_j}(\cdot,s) &\geq c(n,m) \bint_{B(x_o,2r_j)} u_{k_j}(x,\tau) \, \d x \\
&= c(n,m) \eps^\frac{1}{1-m} r_j^{- \frac{2}{1-m}}
\end{align*}
for any $s \in [\tau + \alpha \eps, \tau + \eps]$, where
$\alpha=\alpha(n,m) \in (0,1)$ is the constant from Lemma~\ref{l.wharnack_supercal}. Since the sequence $(r_j)$ and numbers $\tau \in (t_o,t)$ and $\eps \in (0,\min\{c_2 (t-t_o), T-t\} )$ could be chosen freely, the claim follows.
\end{proof}

Next we state characterizations for the complementary class.

\begin{theo} \label{t.complementary}
Let $\frac{n-2}{n} < m < 1$ and $\Omega$ be an open set in $\R^n$. Assume that $u$ is a nonnegative supercaloric function in $\Omega_T$. Then the following statements are equivalent:

\begin{itemize}
\item[(i)] $u\notin L^q_{\loc}(\Omega_T)$ for any $q > \frac{n}{2}(1-m)$,
\item[(ii)] $u \notin L^{\frac{n}{2}(1-m)}_{\loc}(\Omega_T)$,
\item[(iii')] for every $\alpha \in (\frac{n}{2}(1-m),1)$ there exist $\Omega' \Subset \Omega$ and $\delta \in (0,\frac{T}{2})$ such that
$$
\sup_{\delta< t< T-\delta} \int_{\Omega'} u(x,t)^\alpha \, \d x = \infty,
$$
\item[(iii)] there exist $\Omega' \Subset \Omega$ and $\delta \in (0,\frac{T}{2})$ such that
$$
\sup_{\delta< t< T-\delta} \int_{\Omega'} u(x,t) \, \d x = \infty,
	$$
\item[(iv)] There exists $(x_o,t_o) \in \Omega_T$ such that
$$
\liminf_{(x,s) \to (x_o,t)} u(x,s)| x-x_o |^\frac{2}{1-m} > 0
$$
for every $t > t_o$.
\end{itemize}
\end{theo}

\begin{proof}
(ii) $\implies$ (i): H\"older inequality. \\
(iii') $\implies$ (iii): H\"older inequality. \\
(i) $\implies$ (iii'): Elementary. \\
(iv) $\implies$ (iii'): Fix $t > t_o$. Then, for some $r>0$ there exists $\eps > 0$ such that
$$
u(x,s) \left| x-x_o \right|^\frac{2}{1-m} \geq \eps
$$
whenever $(x,s) \in \left( B(x_o,r)\setminus \{x_o\} \right) \times\left(  (t-r,t+r) \setminus \{t_o\} \right) $.

This implies that 
$$
\int_{B(x_o,r)} u(x,t)^\alpha \, \d x = \infty
$$
for every $\alpha \geq \frac{n}{2}(1-m)$ and $t \in \left(  (t-r,t+r) \setminus \{t_o\} \right)$. This implies (iii'). \\
(iv) $\implies$ (ii): Same argument as above. \\
(iii) $\implies$ (iv): By (iii), there exists an instant of time $t_o
\in (0,T)$ and a sequence $(t_j)$ in $(0,T)$ with $t_j \to t_o$, such that 
$$
\lim_{j\to \infty} \int_{\Omega'} u(x,t_j) \, \d x = \infty
$$
for some $\Omega' \Subset \Omega$.

Let us fix a small $r_o >0$. We claim that there exists a point $x_o \in \overline{\Omega'}$ such that 
$$
\lim_{j\to\infty} \int_{B(x_o,r)} u(x,t_j) \, \d x = \infty
$$
for every $r \in (0,r_o)$. This can be shown by contradiction. Assume that for any $y \in \overline{\Omega'}$ there exists a radius $r_y \in (0,r_o)$ such that 
$$
\limsup_{j\to \infty} \int_{B(y,r_y)} u(x,t_j) \, \d x < \infty.
$$
Take an open cover $\{ B(y,r_y) : y \in \overline{\Omega'} \}$ of
$\overline{\Omega'}$. By compactness of $\overline{\Omega'}$, this has a finite subcover, say $\{ B(y_k,r_k) : k = 1,2,...,M \}$, which implies
$$
\int_{\Omega'} u(x,t_j) \, \d x \leq \sum_{k=1}^M \int_{B(y_k,r_k)} u(x,t_j) \, \d x,
$$
for any $j\in\N$. When $j \to \infty$, the left hand side tends to infinity while the right hand side stays bounded implying the desired contradiction. Thus we established that there exists a point $x_o \in \Omega$ such that 
$$
\lim_{j\to \infty} \int_{B(x_o,r)} u(x,t_j) \, \d x = \infty
$$
for arbitrarily small $r > 0$. Now we can use Lemma~\ref{l.blowup} to conclude the proof.
\end{proof}

\section{Pointwise behavior of supercaloric functions} \label{sec:pointwise}

In this section we show that every supercaloric function coincides
with its $\essliminf$-regularization, cf. Theorem~\ref{t.super_lsc} for weak supersolutions. Proofs are partly based
on~\cite{KinnunenLindqvist2006,KinnunenLindqvist_crelle}.

\begin{theo} \label{t.supercal-essliminf}
Let $0 < m< 1$, $\Omega \subset \R^n$ be an open set and $u : \Omega_T \to [0,\infty]$ a supercaloric function in $\Omega_T$. Then,
$$
u(x,t) = \essliminf_{\substack{(y,s) \to (x,t) \\ s<t}} u(y,s)\quad \text{for every } (x,t) \in \Omega_T.
$$
\end{theo}

First we prove existence and properties of a Poisson modification we will use in the proof.

\begin{prop} \label{p.harnack-convergence}
Let $0<m<1$. Let $(h_k)$ be a nondecreasing sequence of continuous weak solutions in $\Omega_T$, i.e. $h_k^m \in L^2_{\loc}(0,T;H^1_{\loc}(\Omega))$ for each $k\in \N$, and suppose that the pointwise limit $\lim_{k\to \infty} h_k = h$ is bounded in $\Omega_T$. Then, $h$ is a locally H\"older continuous weak solution in $\Omega_T$ with $h^m \in L^2(0,T;H^1_{\loc}(\Omega))$, and $\nabla h_k^m \rightharpoonup \nabla h^m$ weakly in $L^2_{\loc}(\Omega_T)$. 
\end{prop}

\begin{proof}
First observe that the sequence $(h_k)$ is bounded since $h_k \leq h$
for every $k\in \N$. By~\cite[Theorem 18.1, Chapter 6]{DGV} it follows
that the family $(h_k)$ is locally equicontinuous. Arzel\'a-Ascoli
theorem implies that there exists a subsequence $h_{k_i}$ that
converges uniformly to some function $g$, which is locally continuous
in $\Omega_T$ by the uniform limit theorem. Furthermore, since $\lim_{k\to \infty} h_k = h$ pointwise, it follows that $g = h$. Lemma~\ref{l.bounded_caccioppoli} implies that $\nabla h_k^m \rightharpoonup \nabla h^m$ weakly in $L^2_{\loc}(\Omega_T)$, which further implies that $h$ is a weak solution in $\Omega_T$ and $h^m \in L^2(0,T;H^1_{\loc}(\Omega))$. As a bounded weak solution $h$ is locally H\"older continuous by~\cite[Theorem 18.1, Chapter 6]{DGV}. 
\end{proof}

\begin{prop} \label{p.poisson-mod}
Let $0<m<1$ and $Q_{t_1,t_2} \Subset \Omega_T$ be a $C^{2,\alpha}$-cylinder. Let $(v_k)$ be a nondecreasing sequence of continuous weak supersolutions in $\Omega_T$ such that $\lim_{k\to \infty} v_k = v$, in which $v$ is a bounded supercaloric function in $\Omega_T$. Then, there exists a Poisson modification defined as
\[ u_P^k =
\begin{cases}
h_k &\text{ in } Q \times (t_1,t_2], \\
v_k &\text{ otherwise},
\end{cases}
\]
where $h_k \in C(\overline{Q_{t_1,t_2}})$ is a weak solution in $Q_{t_1,t_2}$ with $h_k^m \in L^2(t_1,t_2;H^1(Q))$ such that $h_k = v_k$ on $\partial_p Q_{t_1,t_2}$ and $h_k^m - v_k^m \in L^2(t_1,t_2;H^1_0(\Omega))$. Furthermore, $u_P^k$ is nondecreasing and the limit $u_P = \lim_{k\to \infty} u_P^k$ can be written as
\[ u_P =
\begin{cases}
h &\text{ in } Q \times (t_1,t_2], \\
v &\text{ otherwise},
\end{cases}
\]
in which $h \in C(Q_{t_1,t_2})$ is a weak solution in $Q_{t_1,t_2}$ with $h^m \in L^2(t_1,t_2;H^1(Q))$. Moreover, $u_P$ is a bounded supercaloric function in $\Omega_T$ and $\nabla (u_P^k)^m \rightharpoonup \nabla u_P^m$ weakly in $L^2_{\loc}(\Omega_T)$. In particular,
$$
\nabla h_k^m \rightharpoonup \nabla h^m \quad \text{ weakly in } L^2(Q_{t_1,t_2}).
$$
\end{prop}

\begin{proof}
Since $v_k$ is continuous, there exists functions $\psi_k^i \in C^{0,1}(\Omega_T)$ such that $0 \leq \psi_k^i \leq \psi_k^{i+1} \leq v_k^m$ everywhere in $\Omega_T$ for every $i \in \N$, 
$$
\lim_{i\to \infty} \psi_k^i = v_k^m\quad \text{ everywhere in } \Omega_T
$$
and
$$
\sup_{\partial_p Q_{t_1,t_2}}|(\psi_k^i)^\frac{1}{m} - v_k| \xrightarrow{i\to \infty} 0.
$$
Let $h_k^i$ be a weak solution in $Q_{t_1,t_2}$ taking the boundary values $(\psi_k^i)^\frac{1}{m}$ on $\overline{\partial_p Q_{t_1,t_2}}$ both continuously and in Sobolev sense (Theorem~\ref{t.existence_continuous_weaksol}). Denote 
\[ u_P^{k,i} =
\begin{cases}
h_k^i &\text{ in } Q\times(t_1,t_2], \\
v_k &\text{ otherwise}.
\end{cases}
\]
The sequence $h_k^i$ is increasing w.r.t $i$ in $Q_{t_1,t_2}$ by the
aforementioned theorem and $h_k^i \in C(\overline{Q_{t_1,t_2}})$ for
each $i\in\N$. By Theorem~\ref{t.stability_existence}, $h_k^i \xrightarrow{i\to \infty} h_k$ pointwise everywhere in $Q_{t_1,t_2}$, where $h_k \in C(\overline{Q_{t_1,t_2}})$ is a (unique) very weak solution in $Q_{t_1,t_2}$ such that $h_k = v_k$ on $\overline{\partial_p Q_{t_1,t_2}}$. Since the sequence $(h_k^i)$ w.r.t $i$ satisfies the assumptions in
Proposition~\ref{p.harnack-convergence} in $Q_{t_1,t_2}$, we have that
$h_k$ is a locally H\"older continuous weak solution with $h_k^m \in
L^2(t_1,t_2; H^1_{\loc}(Q))$. Then
Theorem~\ref{t.stability_existence} implies $u_P^k \in C
(\Omega_T\setminus(Q\times\{t_2\}))$,
$u_P^{k,i}$ is increasing w.r.t. $i$ and $\lim_{i\to \infty} u_P^{k,i} = u_P^k$ pointwise in $\Omega_T$. Since $v_k$ is supercaloric in $\Omega_T$ by Theorem~\ref{l.weasuper-is-supercal}, we have $u_P^{k,i} \leq v_k \leq v$ everywhere in $\Omega_T$. This implies that also $u_P^k \leq v_k \leq v$.

Next we show that $u_P^k$ is a (bounded) supercaloric function. Since
$u_P^k$ is lower semicontinuous and bounded, properties (i) and (ii)
in Definition~\ref{d.supercal} are clear. For (iii'), let $V_{s_1,s_2}
\Subset \Omega_T$ be a $C^{2,\alpha}$-cylinder and $g \in
C(\overline{V_{s_1,s_2}})$ be a weak solution in $V_{s_1,s_2}$ with $g
\leq u_P^k$ on $\partial_p V_{s_1,s_2}$. Suppose that $V_{s_1,s_2}$
intersects both $Q_{t_1,t_2}$ and its complement since otherwise the
claim is clear. Now since $v_k$ is supercaloric, we immediately have
$g \leq v_k$ in $\Omega_T$, which implies $g \leq u_P^k$ in
$V_{s_1,s_2} \setminus (Q \times (t_1,t_2])$. Since $h_k$ is
supercaloric and $g$ is subcaloric in $Q_{t_1,t_2}$, we can use
Theorem~\ref{l.supersubcal-cylinder-comparison} to conclude that $g
\leq h_k$ in $V_{s_1,s_2} \cap Q_{t_1,t_2}$. Finally, we have $g\leq h_k$ on the slice $(V\cap Q) \times \{t_2\}$ by continuity of $g$ and $h_k$. This implies that $g \leq u_P^k$ in $V_{s_1,s_2}$, which shows that $u_P^k$ is supercaloric.

Since $u_P^k$ is a bounded supercaloric function,
Theorem~\ref{t.bdd_super} implies that $u_P^k \in
L^2_{\loc}(0,T;H^1_{\loc}(\Omega))$. Further $h_k^m \in
L^2(t_1,t_2;H^1(Q))$ since $u_P^k = h_k$ in
$Q_{t_1,t_2}$. Lemma~\ref{l.comparison-weaksupersub} implies that $h_k
\leq h_{k+1} \leq \sup_{\partial_p Q_{t_1,t_2}} v$ for every $k \in
\N$ since the sequence $(v_k)$ is increasing. Since the sequence
$(h_k)$ satisfies the assumptions in Proposition~\ref{p.harnack-convergence} in $Q_{t_1,t_2}$, we have that $h = \lim_{k \to \infty} h_k$ is a locally H\"older continuous weak solution in $Q_{t_1,t_2}$ with $h^m\in L^2(t_1,t_2;H^1_{\loc}(Q))$.

As $(u_P^k)$ is an increasing and uniformly bounded sequence, Lemma~\ref{l.superc_increasing_lim} implies that the limit $u_P$ is a bounded supercaloric function. Furthermore, Theorem~\ref{t.bdd_super} implies that $u_P^m \in L^2_{\loc}(0,T;H^1_{\loc}(\Omega))$. This further implies $h^m \in L^2(t_1,t_2;H^1(Q))$ since $u_P = h$ in $Q_{t_1,t_2}$.

Since $(u_P^k)$ is an increasing, uniformly bounded sequence of weak supersolutions in $\Omega_T$ converging to $u_P$, Lemma~\ref{l.bounded_caccioppoli} implies that $\nabla (u_P^k)^m \rightharpoonup \nabla u_P^m$ weakly in $L^2_{\loc}(\Omega_T)$. This implies that $\nabla h_k^m \rightharpoonup \nabla h^m$ weakly in $L^2(Q_{t_1,t_2})$ since $u_P^k = h_k$ and $u_P = h$ in $Q_{t_1,t_2}$.

\end{proof}

Before a proof of Theorem~\ref{t.supercal-essliminf}, we state and prove another auxiliary result.

\begin{lem} \label{l.supercal_ae_constant}
Let $0 < m< 1$ and $\Omega \subset \R^n$ be a connected open
set. Suppose that $v : \Omega_T \to [0,\infty]$ is a supercaloric
function in $\Omega_T$ and let $Q_{t_1,t_2} \Subset \Omega_T$ such
that $[t_1,t_2]$ is contained in the positivity set
$\Lambda_+$ defined in~\eqref{def:Lambda+}. Assume that
$$
v = \gamma \quad \text{a.e. in } Q_{t_1,t_2}
$$
for some $\gamma \in (0,\infty)$. Then,
$$
v(x,t) = \gamma\quad \text{for every } (x,t) \in Q \times (t_1,t_2].
$$
\end{lem}

\begin{proof}

By lower semicontinuity of $v$ it follows that $v \leq \gamma$ everywhere in $\overline{Q_{t_1,t_2}}$. Thus without loss of generality we may assume that $v$ is bounded in $\Omega_T$. Since $Q_{t_1,t_2} \Subset \Omega \times \Lambda_+$, it follows that there exists $\delta >0$ such that $v > 0$ everywhere in $Q_{t_1,t_2+\delta}$. Let $\psi_k \in C^\infty(Q_{t_1,t_2+\delta})$ such that 
$$
\psi_1 < \psi_2 < ... < v \quad \text{ and } \quad \lim_{k\to \infty} \psi_k = v \quad \text{everywhere in } Q_{t_1,t_2+\delta}.
$$
Now by applying Theorem~\ref{t.obstacle-supercal} in a similar fashion as in Lemma~\ref{l.bdd-positive-supercal-is-supersol}, we can find a sequence of continuous weak supersolutions $v_k$ in $Q_{t_1,t_2+\delta}$ such that $v_1\leq v_2 \leq ... \leq v$ with $\psi_k \leq v_k \leq v$ everywhere in $Q_{t_1,t_2+\delta}$, which implies $v_k(x,t) \to v(x,t)$ for every $(x,t) \in Q_{t_1,t_2+\delta}$. Observe that we further have $\nabla v_k^m \rightharpoonup \nabla v^m$ weakly in $L^2_{\loc}(Q_{t_1,t_2+\delta})$ by Lemma~\ref{l.bounded_caccioppoli}.

Fix $t' \in (t_1,t_2)$ such that $v(x,t') = \gamma$ for a.e. $x \in
Q$. Observe that this holds for a.e. $t' \in (t_1,t_2)$. Furthermore,
fix a $C^{2,\alpha}$-cylinder $Q' \Subset Q$ and define Poisson
modifications of $v_k$ and $v$ in $Q'_{t',t_2+\delta}$ as in Proposition~\ref{p.poisson-mod}.

Since $h_k$ is a weak solution in $Q'_{t',t_2+\delta}$, it follows that
$$
\iint_{Q'_{t',t_2}}-h_k\partial_t\varphi + \nabla h_k^m \cdot \nabla \varphi \, \d x \d t = \int_{Q'} v_k(x,t') \varphi(x,t') \, \d x
$$
for all $\varphi \in C^\infty(Q'_{t',t_2})$ vanishing on the boundary of $Q'_{t',t_2}$ except possibly on $Q' \times \{t'\}$. By Proposition~\ref{p.poisson-mod} we have that $\nabla h_k^m \rightharpoonup \nabla h^m$ weakly in $L^2(Q'_{t',t_2+\delta})$ when $k \to \infty$. Also $v_k(x,t') \xrightarrow{k\to \infty} v(x,t')$ for every $x \in Q'$. Thus by passing to the limit $k \to \infty$ we obtain 
\begin{equation} \label{e.h-wsol}
\iint_{Q'_{t',t_2}}-h\partial_t\varphi + \nabla h^m \cdot \nabla \varphi \, \d x \d t = \int_{Q'} \gamma \varphi(x,t') \, \d x
\end{equation}
since $v(x,t') = \gamma$ for a.e. $x \in Q'$. Since $\gamma > 0$ is a weak solution as a constant, we also have
\begin{equation} \label{e.gamma-wsol}
\iint_{Q'_{t',t_2}}-\gamma \partial_t\varphi + \nabla \gamma^m \cdot \nabla \varphi \, \d x \d t = \int_{Q'} \gamma \varphi(x,t') \, \d x.
\end{equation}
By approximation we may use test functions satisfying $\varphi \in
L^2(t',t_2; H^1_0(Q'))$ with $\partial_t\varphi\in L^2(Q'_{t',t_2})$ and $\varphi(t_2) = 0$. Observe that the Oleinik type test function
\[
\varphi (x,t) := 
\begin{cases}
\int_{t}^{t_2} (v_k^m(x,s) - h_k^m(x,s)) \, \d s, & \mbox{for }t'<t<t_2,\\
0, & \mbox{for }t \geq t_2,
\end{cases}
\]
is admissible. By using this test function and subtracting~\eqref{e.h-wsol} from~\eqref{e.gamma-wsol} we obtain
\begin{align*}
\mathrm I_k &:= \iint_{Q'_{t',t_2}} (\gamma - h) (v_k^m - h_k^m) \, \d x  \d t \\
&= -\iint_{Q'_{t',t_2}} \nabla (\gamma^m - h^m(x,t)) \cdot \int_t^{t_2} \nabla (v_k^m(x,s) - h_k^m(x,s)) \, \d s \, \d x \d t =:  \mathrm{II}_k.
\end{align*}
Observe that since $\nabla v_k^m \rightharpoonup \nabla v^m$ and
$\nabla h_k^m \rightharpoonup \nabla h^m$ weakly in $L^2(Q'_{t',t_2})$
when $k \to \infty$, and $v = \gamma$ a.e. in $Q'_{t',t_2}$,
we obtain
$$
\mathrm{II}_k \xrightarrow{k\to \infty} - \frac12 \int_{Q'} \left| \int_{t'}^{t_2} \nabla (\gamma^m - h^m(x,t))\, \d t \right|^2 \, \d x \leq 0.
$$
Thus, by~\cite[Corollary 3.11]{BLS} and using the facts above we conclude
\begin{align*}
\iint_{Q'_{t',t_2}} |\gamma^m - h^m|^\frac{m+1}{m} \, \d x \d t &\leq \iint_{Q'_{t',t_2}} (\gamma - h) (\gamma^m - h^m) \, \d x \d t \\
&= \lim_{k\to\infty} \mathrm{I}_k = \lim_{k\to\infty} \mathrm{II}_k \leq 0,
\end{align*}
which implies that $h = \gamma$ a.e. in $Q'_{t',t_2}$. Since $h \in C(Q'_{t',t_2+\delta})$, it follows that $h = \gamma$ everywhere in $Q' \times (t',t_2]$. 

Since $h \leq v$ everywhere in $Q'_{t',t_2+\delta}$ and $v \leq
\gamma$ everywhere in $\overline{Q_{t_1,t_2}}$, it follows that
$\gamma = h \leq v\le\gamma$ everywhere in $Q' \times (t',t_2]$, i.e.,
$v = \gamma$ everywhere in $Q' \times (t',t_2]$. Since this holds for
arbitrary $Q' \Subset Q$ and a.e. $t' \in (t_1,t_2)$, the claim follows.

\end{proof}

\textit{Proof of Theorem~\ref{t.supercal-essliminf}. }
Fix $(x_o,t_o) \in \Omega_T$ and denote 
$$
\lambda = \essliminf_{\substack{(y,s) \to (x_o,t_o) \\ s<t_o}} u(y,s).
$$ 
Without loss of generality we may assume that $\Omega$ is connected. By lower semicontinuity of $u$ we have that $\lambda \geq u(x_o,t_o)$. Thus, if $\lambda = 0$ there is nothing to prove. Let us suppose that $\lambda > 0$.

Suppose that also $u(x_o,t_o) > 0$. Then, it follows that $t_o \in
\Lambda_i$ for some $i\in I$, which further implies that there exists $r_o >0$ such that $B_r(x_o) \times (t_o - r^2,t_o) \Subset \Omega \times \Lambda_i$ for every $r < r_o$. This implies that $\lambda > 0$. Furthermore, for any $\gamma \in (0,\lambda)$ there exists $r < r_o$ such that $u \geq \gamma$ a.e. in $B_r(x_o) \times (t_o - r^2,t_o)$. Now $v= \min\{u,\gamma\}$ is a supercaloric function satisfying $v = \gamma$ a.e. in $B_r(x_o) \times (t_o - r^2,t_o)$. By Lemma~\ref{l.supercal_ae_constant}, it follows that $v = \gamma$ everywhere in $B_r(x_o) \times (t_o - r^2,t_o]$, i.e. $u \geq \gamma$ everywhere in $B_r(x_o) \times (t_o - r^2,t_o]$. In particular $u(x_o,t_o) \geq \gamma$. Since $\gamma \leq u(x_o,t_o) \leq \lambda$ and $\gamma \in (0,\lambda)$ was arbitrary, we have $\lambda = u(x_o,t_o)$.

Then suppose that $u(x_o,t_o) = 0$ and $\lambda > 0$. From the latter it follows that there exists $\eps > 0$ and $r>0$ such that 
$$
\essinf_{B_r(x_o) \times (t_o-r^2,t_o)} u \geq \eps.
$$
Thus, 
$$
\essinf_{B_r(x_o)} u(\cdot,t) \geq \eps
$$
for a.e. $t \in (t_o-r^2,t_o)$. Let $(t_i)$ be a sequence in $(t_o-r^2,t_o)$ for which above holds for every $i \in \N$, and $t_i \to t_o$ as $i \to \infty$. Since $u(x_o,t_o) = 0$ implies that $u(x,t_o) = 0$ for all $x \in \Omega$, we may use Lemma~\ref{l.endpoint_vanish} to conclude
$$
0<\eps \leq \essinf_{B_r(x_o)} u(\cdot,t_i) \leq \bint_{B_r(x_o)} u(x,t_i) \, \d x \xrightarrow{i \to \infty} 0,
$$
which is a contradiction. Thus $\lambda = 0$, which completes the proof.
\hfill \qed

\bigskip

In order to summarize our results on the connections between
supercaloric functions and weak supersolutions, we consider the classes 
\begin{align*}
\mathcal{W} &= \{u_* : u \text{ is a weak supersolution in } \Omega_T\}, \\
\mathcal{S} &= \left\{u : u \text{ is a supercaloric function in } \Omega_T \right\}, \\
\mathcal{S}_E &= \left\{u : u \in \mathcal S,\ u^m \in L^2_{\loc}(0,T;H^1_{\loc}(\Omega))\cap L^\frac{1}{m}_{\loc}(\Omega_T) \right\}, \\
\mathcal{W}_b &= \{u_* : u \in \mathcal W, u \text{ is locally essentially bounded in } \Omega_T\},\\
\mathcal{S}_b &= \left\{u : u \in \mathcal S, u \text{ is locally bounded in } \Omega_T  \right\},
\end{align*}
where $(\cdot)_*$ denotes the $\essliminf$-regularization defined in Theorem~\ref{t.super_lsc}.

As a direct consequence of
Lemmas~\ref{l.weasuper-is-supercal},~\ref{l.supercal-in-energyspace}
(or~\ref{t.bdd_super}) and Theorem~\ref{t.supercal-essliminf} together
with the examples presented in Sections~\ref{sec:barenblatt} and~\ref{sec:ipss} we can conclude the following connections of nonnegative supercaloric functions and weak supersolutions.  

\begin{cor}\label{c.connections}
Let $0<m<1$. Then $\mathcal W \subsetneq \mathcal S$, $\mathcal{W} = \mathcal{S}_E$ and $\mathcal{W}_b = \mathcal{S}_b$.
\end{cor}


\begin{thebibliography}{10}

%
%
%
%
%
%
%
%
%
%
%
%
%
%
%
%
%
%
%
%
%
%
%
 
 
 
 
 


  \bibitem{Abdulla1} U.G. Abdulla, \emph{On the Dirichlet problem for the nonlinear diffusion equation in non-smooth domains}, J. Math. Anal. Appl. 260 (2001), no. 2, 384--403.
 
  \bibitem{Abdulla2} U.G. Abdulla, \emph{Well-posedness of the Dirichlet problem for the non-linear diffusion equation in non-smooth domains}, Trans. Amer. Math. Soc. 357 (2005), no. 1, 247--265.

  \bibitem{AL} B. Avelin and T. Lukkari, \emph{Lower semicontinuity of weak supersolutions to the porous medium equation},  Proc. Amer. Math. Soc. 143 (2015), no. 8, 3475--3486.
 
  \bibitem{Bjorns_boundary} A. Bj\"{o}rn, J. Bj\"{o}rn, U. Gianazza and J. Siljander, \emph{Boundary regularity for the porous medium equation},  Arch. Ration. Mech. Anal. 230 (2018), no. 2, 493--538.
 

%

 \bibitem{BDL} V. B\"{o}gelein, F. Duzaar and N. Liao, \emph{On the H\"older regularity of signed solutions to a doubly nonlinear equation}, J. Funct. Anal. 281 (2021), no. 9, Paper No. 109173, 58 pp. 

 \bibitem{BLS} V. B\"{o}gelein, T. Lukkari and C. Scheven, \emph{The obstacle problem for the porous medium equation}, Math. Ann. 363 (2015), no. 1-2, 455--499. 
 
 \bibitem{CV}E. Chasseigne and J. L. V\'azquez, \emph{Theory of extended solutions for fast-diffusion equations in optimal classes of data. Radiation from singularities},  Arch. Ration. Mech. Anal. 164 (2002), no. 2, 133--187.

\bibitem{Cho_Scheven} Y. Cho and C. Scheven, \emph{H\"older regularity for singular parabolic obstacle problems of porous medium type}, Int. Math. Res. Not. IMRN 2020, no. 6, 1671--1717.
 
 
 \bibitem{DK} P. Daskalopoulos and C. E. Kenig, \emph{Degenerate diffusions: Initial value problems and local regularity theory}, EMS Tracts in Mathematics, 1, European Mathematical Society (EMS), Z\"urich, 2007.

 


\bibitem{DGV} E. DiBenedetto, U. Gianazza and V. Vespri, \emph{Harnack's inequality for degenerate and singular parabolic equations}, Springer Monographs in Mathematics, Springer, New York, 2012.  

\bibitem{EG} L. C. Evans and R. F. Gariepy, \emph{Measure theory and fine properties of functions}, Studies in Advanced Mathematics. CRC Press, Boca Raton, FL, 1992. 


\bibitem{GLL} U. Gianazza, N. Liao and T. Lukkari, \emph{A boundary estimate for singular parabolic diffusion equations},  NoDEA Nonlinear Differential Equations Appl. 25 (2018), no. 4, Paper No. 33, 24 pp.

\bibitem{GKM_supercal} R. Kr. Giri, J. Kinnunen and K. Moring, \emph{Supercaloric functions for the parabolic $p$-Laplace equation in the fast diffusion case}, NoDEA Nonlinear Differential Equations Appl. 28 (2021), no. 3, Paper No. 33, 21 pp.




\bibitem{HKM} J. Heinonen, T. Kilpel\"ainen and O. Martio, \emph{Nonlinear potential theory of degenerate elliptic equations}, Oxford Mathematical Monographs. Oxford Science Publications. The Clarendon Press, Oxford University Press, New York, 1993.

\bibitem{KLLP-supercal} J. Kinnunen, P. Lehtel\"a, P. Lindqvist and M. Parviainen, \emph{Supercaloric functions for the porous medium equation}, J. Evol. Equ. 19 (2019), no. 1, 249--270.

\bibitem{KinnunenLindqvist_crelle} J. Kinnunen and P. Lindqvist, \emph{Definition and properties of supersolutions to the porous medium equation}, J. Reine Angew. Math. 618 (2008), 135--168. 


\bibitem{KinnunenLindqvist2006} J. Kinnunen and P. Lindqvist, \emph{Pointwise behaviour of semicontinuous supersolutions to a quasilinear parabolic equation},  Ann. Mat. Pura Appl. (4) 185 (2006), no. 3, 411--435.

\bibitem{KLL} J. Kinnunen, P. Lindqvist and T. Lukkari, \emph{Perron's method for the porous medium equation}, J. Eur. Math. Soc. (JEMS) 18 (2016), no. 12, 2953--2969. 


\bibitem{KoKuPa} R. Korte, T. Kuusi and M. Parviainen, \emph{A connection between a general class of superparabolic functions and supersolutions}, J. Evol. Equ. 10 (2010), no. 1, 1--20.




%
%

\bibitem{KuLiPa} T. Kuusi, P. Lindqvist, and M. Parviainen, \emph{Shadows of infinities},  Ann. Mat. Pura Appl. (4) 195 (2016), no. 4, 1185--1206.


\bibitem{Pekka_harnack} P. Lehtel\"a, \emph{A weak Harnack estimate for supersolutions to the porous medium equation}, Differential Integral Equations 30 (2017), no. 11-12, 879--916.



\bibitem{Naian} N. Liao, \emph{Regularity of weak supersolutions to elliptic and parabolic equations: lower semicontinuity and pointwise behavior}, J. Math. Pures Appl. (9) 147 (2021), 179--204.

\bibitem{lukkari-fde-measuredata} T. Lukkari, \emph{The fast diffusion equation with measure data}, NoDEA Nonlinear Differential Equations Appl. 19 (2012), no. 3, 329--343.

\bibitem{MSb} K. Moring and L. Sch\"atzler, \emph{Continuity up to the boundary for obstacle problems to porous medium type equations}, arXiv, 2023.

 \bibitem{MS} K. Moring and L. Sch\"atzler, \emph{On the H\"older regularity for obstacle problems to porous medium type equations}, J. Evol. Equ. 22 (2022), no. 4, Paper No. 81, 46 pp.

\bibitem{MSo} K. Moring and C. Scheven, \emph{On two notions of solutions to
  the obstacle problem for the singular porous medium equation},
arXiv, 2023.
   
\bibitem{S-existence}
  L.~Sch\"atzler,
  \emph{Existence for singular doubly nonlinear systems of porous medium type with time dependent boundary values}, J. Elliptic Parabol. Equ. 5 (2019), no. 2, 383--421.  
  
  \bibitem{Schaetzler2}
L.~Sch\"atzler,
\emph{The obstacle problem for singular doubly nonlinear equations of porous medium type}, Atti Accad. Naz. Lincei Rend. Lincei Mat. Appl. 31 (2020), no. 3, 503--548. 


 \bibitem{Vazquez2} J. L. V\'azquez, \emph{Smoothing and decay estimates for nonlinear diffusion equations. Equations of porous medium type}, Oxford Lecture Series in Mathematics and its Applications, 33. Oxford University Press, Oxford, 2006.

 \bibitem{Vazquez} J. L. V\'azquez, \emph{The porous medium equation: Mathematical theory}, Oxford Mathematical Monographs, The Clarendon Press, Oxford University Press, Oxford, 2007.
%
%
%
%
 


\end{thebibliography}
\end{document}